\documentclass{smfart}

\usepackage{fancyhdr} 

\usepackage[francais]{babel}
\usepackage[latin1]{inputenc}
\usepackage[T1]{fontenc}
\usepackage{float, graphicx}

\usepackage{amsmath,amssymb,bbm}

\usepackage{color}

\usepackage{stmaryrd}
\usepackage{hyperref}
\usepackage{hypcap}

\newtheorem{thm}{Th\'eor\`eme}[section]
\newtheorem{prop}[thm]{Proposition}
\newtheorem{lemme}[thm]{Lemme}

\newtheorem{props}[thm]{Propri\'et\'es}
\newtheorem{define}[thm]{D\'efinition}
\newtheorem{cor}[thm]{Corollaire}

\newtheorem{ex}[thm]{Exemple}
\newtheorem{rem}[thm]{Remarque}
\newtheorem{conj}[thm]{Conjecture}
\newtheorem{quest}{Question}

\theoremstyle{plain}
\AtBeginDocument{\let\oldtheo\thm
\renewcommand{\thm}{%
\oldtheo\hypertarget{\***@currentHref\endcsname}{}}}

\newcommand\M{\mathbb M}
\newcommand\N{\mathbb N}
\newcommand\Z{\mathbb Z}
\newcommand\Q{\mathbb Q}
\newcommand\R{\mathbb R}
\newcommand\A{\mathcal A}
\newcommand\D{\mathcal D}
\newcommand\transp[1]{\, {}^t \! {#1}}
\newcommand\transpA{\, {}^t \!\! A}
\newcommand\abs[1]{|#1|}

\newcommand\eq{\Leftrightarrow}
\newcommand\imp{\Rightarrow}

\newcommand\longeq{\Longleftrightarrow}

\newcommand\m[1]{T_{#1}}
\newcommand\floor[1]{\left\lfloor #1 \right\rfloor}
\newcommand\ceil[1]{\left\lceil #1 \right\rceil}
\newcommand\discr{\operatorname{discr}}
\newcommand\Tr{\operatorname{Tr}}
\newcommand\Norm{\operatorname{N}}
\newcommand\Det{\operatorname{Det}}
\newcommand\pgcd{\operatorname{pgcd}}
\newcommand\pat[1]{\textsf{#1}}

\newcommand\pattranspA{\, {}^t \! \pat{A}}

\newcommand\defi[1]{\textsl{#1}}

\begin{document}

\title{Construction de fractions continues p\'eriodiques uniform\'ement born\'ees}

\author{Paul MERCAT}

\address{Université Paris-Sud 11 \\ 91405 Orsay}
\email{paul.mercat@math.u-psud.fr}
\urladdr{http://www.math.u-psud.fr/~mercat}

\date{\today}

\begin{abstract}
	Nous construisons, dans les corps quadratiques réels, une infinité de fractions continues périodiques uniformément bornées, avec une borne qui semble meilleure que celle connue jusqu'ici.
	Nous faisons cela en partant de développements en fractions continues de la même forme que ceux des réels $\sqrt{n} + n$. Et ceci nous permet d'obtenir de plus qu'il existe une infinité de corps quadratiques contenant une infinité de développements en fractions continues périodiques formées seulement des entiers $1$ et $2$.
	Nous montrons aussi qu'une conjecture de Zaremba implique une conjecture de McMullen, en construisant des fractions continues périodiques à partir de développements en fractions continues de rationnels. 
\end{abstract}

\maketitle

\tableofcontents

\section{Introduction} \label{intro}
Il est bien connu qu'un r\'eel est quadratique sur $\Q$ si et seulement si son d\'eveloppement en fraction continue est p\'eriodique \`a partir d'un certain rang.
Il est par contre remarquable que dans un corps quadratique donn\'e il existe une infinit\'e de fractions continues p\'eriodiques uniform\'ement born\'ees.
Par exemple, $\Q[\sqrt{10}]$ contient la suite de fractions continues p\'eriodiques
$$[\overline{1,1,2,1,1,2}]$$
$$[\overline{1,1,2,1,\textcolor{red}{2,1,1},1,2,\textcolor{red}{2,1,1}}]$$
$$[\overline{1,1,2,1,\textcolor{red}{2,1,1}, \textcolor{red}{2,1,1},1,2,\textcolor{red}{2,1,1}, \textcolor{red}{2,1,1}}]$$
$$[\overline{1,1,2,1,\textcolor{red}{2,1,1},\textcolor{red}{2,1,1},\textcolor{red}{2,1,1},1,2,\textcolor{red}{2,1,1},\textcolor{red}{2,1,1},\textcolor{red}{2,1,1}}]$$
$$...$$
qui est uniform\'ement born\'ee par 2. Et Wilson a démontré dans \cite{wilson} qu'il y avait des suites semblables dans tous les corps quadratiques réels. Mais peut-on obtenir une borne ind\'ependante du corps quadratique ?
Par exemple, peut-on trouver dans n'importe quel corps quadratique, un r\'eel quadratique dont le d\'eveloppement en fraction continue n'a que des $1$ et des $2$ ?
Ces questions sont ouvertes.
Dans son papier \cite{mcmullen}, McMullen s'intéresse à de telles suites de fractions continues, et redémontre le résultat de Wilson en donnant d'autres exemples de telles suites. Il établit des liens avec des questions sur les géodésiques fermées de certaines variétés arithmétiques, ainsi qu'avec des questions de théorie des nombres sur le comptage d'idéaux.

Dans cet article, nous nous int\'eressons à ces suites de fractions continues p\'eriodiques
qui restent dans un corps quadratique donn\'e. Nous allons voir en particulier que dans une infinit\'e de corps quadratiques il existe une infinité de fractions continues p\'eriodiques form\'ees seulement des entiers $1$ et $2$.

\subsection{Fractions continues}
Tout r\'eel $x$ peut se d\'evelopper en fraction continue
$$x = [a_1, a_2, a_3, ... ] = a_1 + \cfrac{1}{a_2+\cfrac{1}{a_3+\cfrac{1}{...}}}$$
avec des quotients partiels $a_i \in \Z$ et $a_i \geq 1$ si $i \geq 2$. Si le d\'eveloppement est p\'eriodique \\ \mbox{(c'est-\`a-dire s'il existe un entier $p$ tel que pour tout $i$, on ait $a_i = a_{i+p}$)}, \\
on notera $x = [\overline{a_1, a_2, ..., a_{p}}]$.

\begin{define}
On appellera \defi{quasi-palindromique} une fraction continue de la forme $[\overline{a_0, a_1, a_2, a_3, ... , a_3, a_2 , a_1}]$.
\end{define}
Dans cette d\'efinition, la partie sym\'etrique $a_1, a_2, ... , a_2, a_1$ peut ou non avoir un terme m\'edian.
Dans la suite, la notation $(a_0, a_1, a_2, ... , a_{k-1}, a_k)^n$ signifie que le motif $a_0, a_1, a_2, ... , a_{k-1}, a_k$ est r\'ep\'et\'e $n$ fois.

Pour des références sur les fractions continues, on pourra consulter par exemple~\cite{perron}, \cite{hp}, \cite{schmidt}, ou encore~\cite{bugeaud}.

Dans le chapitre~\ref{preuve}, nous d\'emontrons le r\'esultat nouveau suivant :

\begin{thm} \label{MN}

Si la fraction continue p\'eriodique quasi-palindromique $$[\overline{a_0, a_1, a_2, ... , a_2 , a_1}]$$
est dans $\Q[\sqrt{\delta}]$, alors il existe deux uplets d'entiers strictement positifs \\ $(b_1, b_2, ... , b_k)$ et $(c_1, c_2, ... , c_l)$ tels que $\Q[\sqrt{\delta}]$ contienne la suite non constante de fractions continues p\'eriodiques
$$[\overline{b_1, b_2, ... , b_k, (a_0, a_1, a_2, ... , a_2, a_1)^n, c_1, c_2, ... , c_l, (a_1, a_2, ... , a_2, a_1, a_0)^n}]$$


En outre, si $m$ est un majorant des entiers $a_i$, alors on peut demander à ce que $2m+1$ soit un majorant des entiers $b_i$ et $c_i$.

\end{thm}


Cela permet de retrouver le résultat de Wilson \cite{wilson} : 

\begin{cor}
Pour tout corps quadratique $\Q[\sqrt{\delta}]$, il existe un réel $m_\delta$ et une infinit\'e de fractions continues p\'eriodiques $[ \overline{ a_0, a_1, ..., a_n } ] \in \Q[\sqrt{\delta}]$ avec $1 \leq a_i \leq m_{\delta}$.
\end{cor}
Ici, $m_{\delta}$ d\'epend seulement de $\delta$. Par exemple on peut prendre $m_{10} = 2$ d'apr\`es la suite annonc\'ee au d\'ebut.
Nous d\'emontrons dans le chapitre~\ref{exs}, que l'on peut m\^eme trouver dans tout corps quadratique r\'eel, une infinit\'e de fractions continues p\'eriodiques ne comportant que trois quotients partiels diff\'erents.

En appliquant le th\'eor\`eme aux r\'eels $\sqrt{\delta}+\floor{\sqrt{\delta}}$, on obtient que l'on peut prendre $m_{\delta} = 4\sqrt{\delta}+1$, ce qui am\'eliore le r\'esultat de Wilson qui dit que l'on peut prendre $m_\delta = O(\delta)$ (voir le paragraphe \ref{rqp} sur les r\'eels quasi-palindromiques).

La preuve du th\'eor\`eme permettra aussi d'obtenir le r\'esultat nouveau :

\begin{thm}
Il existe une infinit\'e de corps quadratiques $\Q[\sqrt{\delta}]$ dans lesquels il existe une infinit\'e de fractions continues p\'eriodiques $[ \overline{ a_0, a_1, ..., a_n } ]$ avec $a_i \in \{ 1, 2 \}$.
\end{thm}

Dans son article \cite{mcmullen}, McMullen émet la conjecture suivante :

\begin{conj}[McMullen] \label{conjMcMullen}
Dans tout corps quadratique r\'eel, il existe une infinit\'e de fractions continues p\'eriodiques form\'ees seulement des entiers 1 et 2.
\end{conj}

Les deux cas particuliers suivants semblent aussi ouverts :

\begin{conj} \label{conj12} 
Dans tout corps quadratique r\'eel, il existe une fraction continue p\'eriodique form\'ee seulement des entiers 1 et 2.
\end{conj}

\begin{rem}
La conjecture \ref{conj12} est vraie pour tous les corps quadratiques $\Q[\sqrt{\delta}]$ pour $\delta < 127$.
Par exemple, $\Q[\sqrt{19}]$ contient la fraction continue p\'eriodique de longueur 78 suivante :
$$ [ \overline{2,1,1,1,\pat{M},1,1,1,2,\transp{\pat{M}}} ]  \text{\ o\`u\ } \pat{M} \mathrm{\ est\ le\ motif\ suivant\ :}$$
$$(1,1,1,1,1,1,1,1,1,2,1,1,2,2,1,2,1,1,2,1,1,1,1,2,2,1,1,1,2,1,1,2,1,2,2)$$
et $\transp{\pat{M}}$ est le motif miroir. 
\end{rem}

Nous avons vérifié cette remarque par ordinateur. Voir la remarque \ref{r_conjMcM} pour plus de détails.

\begin{rem}
	Dans~\cite{jp}, Jenkinson et Pollicott étudient l'ensemble $E_2$ des réels dont le développement en fraction continue est infini et ne s'écrit qu'avec les nombres $1$ et $2$.
	La conjecture~\ref{conj12} se reformule comme suit : \\L'intersection de l'ensemble $E_2$ avec tout corps quadratique réel $\Q[\sqrt{\delta}]$ est non vide.
\end{rem}

\begin{conj} \label{conjM}
Il existe un entier $m$ tel que tout corps quadratique r\'eel $\Q[\sqrt{\delta}]$ contienne une infinit\'e de fractions continues p\'eriodiques uniform\'ement born\'ees par $m$.
\end{conj}

%
Nous avons obtenu un le lien entre cette conjecture sur les fractions continues périodiques et la conjecture de Zaremba suivante sur les fractions continues finies :

\begin{conj}[Zaremba] \label{conjZar}
	Il existe une constante $m$ telle que pour tout entier $q \geq 1$, il existe un entier $p$ premier à $q$ tel que l'on ait
	\[ \frac{p}{q} = [a_0, a_1, a_2, \dots ] \]
	où les entiers $a_i$ sont entre $1$ et $m$.
\end{conj}

Voici un lien entre ces deux conjectures :

\begin{thm}
	La conjecture de Zaremba~\ref{conjZar} implique la conjecture~\ref{conjM}.
\end{thm}

Nous obtenons ce dernier résultat comme corollaire du théorème suivant, qui permet de construire des fractions continues périodiques dans un corps donné à partir de fractions continues de certain rationnels :

\begin{thm} \label{thm_fcfp}
	Soient $a$, $b$, $c$ et $\delta$ des entiers strictement positifs tels que
	\begin{itemize}
		\item $b$ et $c$ sont solution de l'équation de Pell-Fermat : $c^2 - \delta b^2 = \pm 1$,
		\item $a$ et $c$ sont premiers entre eux et $a < c$.
	\end{itemize}
	Alors on a l'une des égalités
	\[
		\frac{c-a + b \sqrt{\delta}}{c} = [\overline{1,1,a_1-1, a_2, a_3, \dots, a_{n-1}, a_n, 1, 1, a_n - 1, a_{n-1}, a_{n-2}, \dots, a_2, a_1}]
	\]
	ou
	\[
		\frac{c-a + b \sqrt{\delta}}{c} = [\overline{1,1,a_1-1, a_2, a_3, \dots, a_{n-1}, a_n - 1, 1, 1, a_n, a_{n-1}, a_{n-2}, \dots, a_2, a_1}],
	\]
	où $[0,a_1, a_2, \dots, a_{n-1}, a_n]$ est le développement en fraction continue du rationnel $\frac{a}{c}$.
\end{thm}

Si l'on a des entiers nuls dans la fraction continue donnée par ce théorème (c'est le cas si $a_1 = 1$ ou $a_n = 1$), il suffit de remplacer chaque triplet $x, 0, y$ par $x+y$ pour obtenir une vraie fraction continue. Remarquons aussi que les entiers $\delta$ qui satisfont les hypothèses de ce théorème sont nécessairement non carrés.\\

\`A la vue du th\'eor\`eme~\ref{MN}, on peut se poser la question suivante, qui impliquerait la conjecture~\ref{conjM} :

\begin{quest}
Existe-t-il une constante $m$ telle que dans tout corps quadratique r\'eel, il existe une fraction continue p\'eriodique de la forme $ [ \overline{ a_0, a_1, a_2, ... , a_2, a_1 } ] $ uniform\'ement born\'ee par $m$ ?
\end{quest}

Cependant, m\^eme sans la condition sur la forme de la fraction continue p\'eriodique, le probl\`eme semble ouvert :

\begin{conj} \label{mcmullen1}
Il existe une constante $m$ telle que dans tout corps quadratique r\'eel, il existe une fraction continue périodique uniform\'ement born\'ee par $m$.
\end{conj}

Nous obtenons une petite avancée dans la résolution de cette conjecture, comme corollaire du théorème~\ref{thm_fcfp}, en utilisant les travaux de Bourgain et Kontorovich sur la conjecture de Zaremba.
Voir corollaire~\ref{thm_ibk}. \\


\noindent Voici le plan de cet article :

Dans le chapitre~\ref{pos} qui suit, on rappelle quelques propri\'et\'es des fractions continues et du semi-groupe des matrices positives.
Dans le chapitre~\ref{preuve} on donne une preuve rapide du th\'eor\`eme~\ref{MN}.
Le chapitre~\ref{suites1} explique pourquoi les suites de fractions continues p\'eriodiques de la forme $[\overline{b_1,b_2, ..., b_k,(a_1, a_2, ... , a_l)^n}]$ ne conviennent pas.
Les chapitres~\ref{BACtA} et~\ref{suites2} sont consacr\'es aux g\'en\'eralisations et r\'eciproques des r\'esultats qui nous ont permis d'obtenir le th\'eor\`eme~\ref{MN}. Et enfin, le chapitre~\ref{exs} explicite les suites de fractions continues p\'eriodiques que permet d'obtenir la preuve du th\'eor\`eme~\ref{MN}, et donne des exemples. 
Nous finissons par le chapitre~\ref{s_zaremba}, dans lequel nous démontrons le théorème~\ref{thm_fcfp} et montrons que la conjecture de Zaremba~\ref{conj_zaremba}
implique la conjecture~\ref{conjM}. 

Ce travail a été principalement réalisé pendant mon stage de Master, à l'université Pierre et Marie Curie, à Paris, en 2009.
Je remercie Yves Benoist, qui était alors mon encadrant, pour l'aide précieuse qu'il m'a apporté.
Je remercie aussi Yann Bugeaud pour ses remarques qui m'ont aidées à trouver le théorème \ref{thm_fcfp}.

\section{Matrices positives} \label{pos}

Dans cette section, on s'intéresse au semi-groupe $\Gamma$ des matrices positives, qui est l'ensemble des matrices de $GL_2(\Z)$ qui correspondent à des développements en fraction continue périodique. On démontre des propriétés qui serviront par la suite.

\begin{define}
	Soit $M$ une matrice de $M_2(\R)$.
	On appelle \defi{discriminant} de $M$ le discriminant du polyn\^ome caract\'eristique de $M$.
	Autrement dit, c'est le r\'eel 
	\[ \discr(M) := (d-a)^2 + 4bc = \Tr(M)^2 - 4\Det(M), \]
	pour $M = \left( \begin{array}{cc} a & b \\ c & d \end{array} \right)$.
\end{define}

\begin{props}
	Soit $M$ une matrice de $GL_2(\Z).$
	\begin{enumerate}
		\item Les valeurs propres de $M$ sont dans le corps $\Q[\sqrt{\discr(M)}]$.
		\item Les vecteurs propre de $M$ peuvent \^etre choisis \`a coefficients dans le corps $\Q[\sqrt{\discr(M)}]$.
		\item Pour tout $n \geq 1$, $\discr(M)$ divise $\discr(M^n)$ et $\frac{\discr(M^n)}{\discr(M)}$ est un carr\'e parfait.
	\end{enumerate}
\end{props}

\begin{proof}
	\defi{1 .} Les valeurs propres sont les racines du polyn\^ome caract\'eristique, donc sont dans $\Q[\sqrt{\discr(M)}]$. \\
	\defi{2 .} Si l'on choisis un vecteur propre de la forme $\left( \begin{array}{c} 1 \\ x \end{array} \right)$, alors $x$ est racine du polyn\^ome $bx^2 + (a-d)x - c = 0$, pour $M = \left( \begin{array}{cc} a & b \\ c & d \end{array} \right)$. Or ce polyn\^ome a pour discriminant $\discr(M)$, et donc $x$ est dans $\Q[\sqrt{\discr(M)}]$. \\
	\defi{3 .} Si $\lambda$ et $\mu$ sont les deux valeurs propres de $M$, alors $\discr(M^n) = (\lambda^n - \mu^n)^2$. On a donc
	\[ \frac{\discr(M^n)}{\discr(M)} = (\frac{\lambda^n - \mu^n}{\lambda - \mu})^2 = (\sum_{i = 0}^{n-1}{\lambda^i\mu^{n-i-1}})^2. \]
	Or, l'expression $\sum_{i = 0}^{n-1}{\lambda^i\mu^{n-i-1}}$ est un polyn\^ome  \`a coefficients entiers, sym\'etrique en $\lambda$ et $\mu$. C'est donc aussi un polyn\^ome \`a coefficients entiers en la trace et le d\'eterminant de $M$. Donc le nombre complexe $\sum_{i = 0}^{n-1}{\lambda^i\mu^{n-i-1}}$ est en fait un entier.
\end{proof}

\subsection{Notations}

\begin{itemize}
	\item On appelle \defi{corps d'une matrice} $M \in GL_2(\Z)$ le corps $\Q[\sqrt{\discr(M)}]$. 
	\item On note $\m{i}$ la matrice $\left(
		\begin{array}{cc}
		0 &1\\
		1 &i
		\end{array}
		\right)$ o\`u $i$ est un entier,
		et on note $\m{(a_1, a_2, ... , a_n)}$ le produit $\m{a_1}\m{a_2} ... \m{a_n}$ pour un $n$-uplet d'entiers $(a_1, a_2, ..., a_n)$. Cette notation est justifiée par la remarque \ref{rem_mat-fc}.
	\item On note $\Gamma$ l'ensemble des produits des matrices $\m{i}$ pour $i \geq 1$
		($\Gamma$ est donc le monoïde engendr\'e par les matrices $\m{i}$, et il contient $I_2$), et $\Gamma_n$ l'ensemble des produits de matrices $\m{i}$ pour $1 \leq i \leq n$.
		On dit qu'une matrice est \defi{positive} si c'est un élément du semi-groupe $\Gamma\backslash\{I_2\}$.
		En particulier, une matrice positive a tous ses coefficients positifs ou nuls.

	\item Etant donn\'ee une matrice $P \in M_2(\R)$, on note $P^{\dagger}$ l'unique matrice v\'erifiant $P + P^{\dagger} = \Tr(P)I_2$. Si $P = \left(
		\begin{array}{cc}
		a &b\\
		c &d
		\end{array}
		\right)$, on a donc $P^{\dagger} = \left(
		\begin{array}{cc}
		d &-b\\
		-c &a
		\end{array}
		\right)$. Et si $P$ est inversible, on peut \'ecrire $P^{\dagger} = \Det(P)P^{-1}$.
	
		Les relations $P^{\dagger}Q^{\dagger} = (QP)^{\dagger}$ et $PP^{\dagger} = \Det(P)I_2$ permettent d'obtenir l'\'egalit\'e bien utile :
		\begin{align} \label{fm}
			\Det(P+Q) = \Det(P) + \Det(Q) + \Tr(PQ^{\dagger})
		\end{align}
	\item On appelle \defi{longueur} d'une fraction continue p\'eriodique la plus petite p\'eriode.
\end{itemize}

\begin{prop}[Crit\`ere de positivit\'e] \label{carGamma} 
	Soit $M = \left(
	\begin{array}{cc}
	a &b\\
	c &d
	\end{array}
	\right) \in GL_2(\Z)$.\\
	Alors les propri\'et\'es suivantes sont \'equivalentes:
	\begin{enumerate}
		\item $M \in \Gamma \backslash \{ I_2 \}$ (i.e. $M$ est positive) \label{gpos}
		\item$0 \leq a \leq b \leq d$ et $a \leq c \leq d$ \label{ppt-pos}
		\item Il existe un r\'eel quadratique $x_M > 1$, dont le conjugu\'e $\overline{x_M}$ est dans $]-1,0[$, tel que $\left( \begin{array}{c} 1 \\ x_M \end{array} \right)$ et $\left( \begin{array}{c} 1 \\ \overline{x_M} \end{array} \right)$ soient des vecteurs propres de $M$, et les entiers $b$ et $d$ sont strictement positifs. \label{vpos}
		\item $\abs{b-c} < d-a$ et les entiers $a, b, c$ et $d$ sont positifs ou nuls. \label{ipos}
	\end{enumerate}
\end{prop}

\begin{proof}

	\underline{$\ref{gpos} \imp \ref{ppt-pos}$}
	R\'ecurrence facile.
	
	\underline{$\ref{ppt-pos} \imp \ref{gpos}$}
	Supposons que $M$ v\'erifie ces in\'egalit\'es.
	D\'emontrons que $M$ est positive par r\'ecurrence sur ses coefficients.
	
	Si $a = 0$, alors $bc = -\Det(M) = \pm 1$, donc $b = c = 1$ et $M = \begin{pmatrix} 0 & 1 \\ 1 & d \end{pmatrix} = \m{d}$.
	
	Si $a = 1$, alors $M$ est de la forme
	$$\left.	\begin{array}{l}
				M = \left( \begin{array}{cc} 1 & b \\ c & bc+1 \end{array} \right) = \m{c}\m{b}$ quand $\Det(M) = 1, \\
				M = \left( \begin{array}{cc} 1 & b \\ c & bc-1 \end{array} \right) = \m{c-1}\m{1}\m{b-1}$ quand $\Det(M) = -1.
			\end{array} \right.$$
	
	Si $a \geq 2$, alors on a $\floor{\cfrac{b}{a}} = \floor{\cfrac{d}{c}}$ (cela d\'ecoule de l'\'egalit\'e $ad-bc = \pm 1$).
	Et en posant $q = \floor{\cfrac{b}{a}} = \floor{\cfrac{d}{c}}$, la matrice $M \m{q}^{-1} = \left(
	\begin{array}{cc}
	b-qa &a\\
	d-qc &c
	\end{array}
	\right)$ v\'erifie encore le point~\ref{ppt-pos} de la proposition : l'in\'egalit\'e $b - qa \leq d-qc$ r\'esulte de $\Det(M \cdot \m{q}^{-1}) = \pm 1$ et de $2 \leq a \leq c$, et les autres in\'egalit\'es sont claires.
	De plus, la matrice $M \m{q}^{-1}$ est strictement plus petite que $M$, coefficients \`a coefficients. 
	
	\underline{$\ref{ppt-pos} \imp \ref{ipos}$}
	Clair.
	
	\underline{$\ref{vpos} \eq \ref{ipos}$}
	Les r\'eels quadratiques $x_M$ et $\overline{x_M}$ sont les racines du polyn\^ome $Q(X) = bX^2+(a-d)X-c$.
	Si $b$ est positif, on a donc,
	\begin{eqnarray} \label{equiv34}
		 \left\{ \begin{array}{c} x_M > 1 \\  -1 < \overline{x_M} < 0 \end{array} \right. \longeq \left\{ \begin{array}{c} Q(1) < 0 \\  Q(-1) > 0 \\ Q(0) < 0 \end{array} \right. \longeq \left\{ \begin{array}{c} \abs{b-c} < d-a \\  c > 0 \end{array} \right.
	\end{eqnarray}
	Et l'\'egalit\'e $bc = 0$ est impossible si $\abs{b-c} < d-a$, puisque si elle avait lieu, alors l'\'egalit\'e $ad - bc = \pm 1$ et la positivit\'e de $a$ et de $d$ entra\^ineraient que $a=d = 1$. 
	
	\underline{$\ref{ipos} \imp \ref{vpos}$}
	On a $b > 0$ et $c > 0$, parce que $bc$ est non nul.
	On peut donc utiliser l'\'equivalence (\ref{equiv34}) ci-dessus pour obtenir $x_M > 1$ et $ -1 < \overline{x_M} < 0$.
	On a ensuite $\abs{b-c} < d-a$ qui entra\^ine $d > 0$.
	
	\underline{$\ref{vpos} \imp \ref{ipos}$}
	L'\'equivalence (\ref{equiv34}) ci-dessus nous donne $\abs{b-c} < d-a$ et  $c > 0$.
	L'\'egalit\'e $ad-bc = \pm 1$ et l'in\'egalit\'e $bc > 0$ donnent alors que $ad \geq 0$, d'o\`u $a \geq 0$.
	
	\underline{$\ref{ipos} \imp \ref{ppt-pos}$}
	Si l'on avait $c < a$, on aurait $cd <  ad = bc \pm 1,$
	donc $cd \leq bc$ puis $d \leq b$.
	Mais la différence $d-a < b - c$ des inégalités contredirai alors l'hypothèse $\abs{b - c} < d-a$.
	Les autres inégalités $a \leq b$, $c \leq d$ et $b \leq d$ s'obtiennent de la même façon.
\end{proof}

\begin{rem} \label{rem_mat-fc}
	Le r\'eel $x_M$ de la proposition ci-dessus qui correspond \`a une matrice
	\[ M = \m{(a_1, a_2, a_3, ..., a_p)} = \m{a_1}\m{a_2}\m{a_3} ... \m{a_p}, \]
	a pour d\'eveloppement en fraction continue
	\[ x_M = [ \overline{ a_1, a_2, a_3, ... , a_p } ]. \]
	La correspondance est bijective si l'entier $p$ est la longueur.
\end{rem}

\begin{define} \label{def-triv}
	Une suite $(A_n)$ de matrices est non triviale si l'ensemble des r\'eels $x_{A_n}$ est infini (ce qui signifie simplement que la suite de matrices correspond \`a une infinit\'e de fractions continues p\'eriodiques).
\end{define}

\begin{rem}
	Il y a unicit\'e de l'\'ecriture d'une matrice de $\Gamma$ comme produit de matrices $\m{i}$. En particulier, une matrice de $\Gamma$ est sym\'etrique si et seulement si son \'ecriture comme produit de matrices $\m{i}$ est sym\'etrique.
	
	Les matrices positives sont toujours diagonalisables et ont des valeurs propres r\'eelles distinctes.
\end{rem}

Pour démontrer l'unicité de la décomposition $M = \m{i_1} \m{i_2} ... \m{i_n}$ d'une matrice positive 
$
	M =	\begin{pmatrix}
			a & b \\
			c & d
		\end{pmatrix}
$,
il suffit de remarquer que la dernière matrice $\m{i_n} = \begin{pmatrix} 0 & 1 \\ 1 & i_n \end{pmatrix}$ s'obtient par la formule 
\[
	i_n = \min(\floor{\frac{b}{a}}, \floor{\frac{d}{c}}),
\]
avec la convention $\floor{\frac{b}{0}} = \infty$. Voir la preuve de la proposition \ref{carGamma} pour plus de détails. \\

On dispose \'egalement d'un crit\`ere de positivit\'e pour les matrices de rang 1 :

\begin{prop} \label{H->Gamma}
	Soit $H = \left(
	\begin{array}{cc}
	a &b\\
	c &d
	\end{array}
	\right) \in M_2(\Z)$ une matrice de rang 1 v\'erifiant les in\'egalit\'es $0 \leq a \leq b \leq d$ et $a \leq c \leq d$.
	Alors il existe deux matrices $P$ et $Q$ dans $\Gamma$ et un entier $e \geq 1$ tels que $H = P \left(
	\begin{array}{cc}
	0 &0\\
	0 &e
	\end{array}
	\right) Q$.
	De plus, $e$ est unique, et $P$ et $Q$ sont uniques modulo les relations
	$$ \left. \begin{array}{c}
		X \m{n}\m{1} \left(
		\begin{array}{cc}
			0 &0\\
			0 &e
		\end{array}
		\right)Y = X \m{n+1}  \left(
		\begin{array}{cc}
			0 &0\\
			0 &e
		\end{array}
		\right) Y \\
		X \left(
		\begin{array}{cc}
			0 &0\\
			0 &e
		\end{array}
		\right) \m{1}\m{n} Y = X \left(
		\begin{array}{cc}
			0 &0\\
			0 &e
		\end{array}
		\right) \m{n+1} Y
	\end{array} \right.$$
	pour $X, Y \in \Gamma$.
\end{prop}

\begin{proof}
	On a n\'ecessairement $e = \pgcd(a, b, c, d)$, d'o\`u son unicit\'e.
	On peut supposer $e = 1$ quitte \`a diviser $H$ par $e$.
	Comme $H$ est non inversible, il existe des entiers positifs ou nuls $x$, $y$, $z$ et $t$ tels que $H = \left(
	\begin{array}{c}
	x\\
	y
	\end{array}
	\right) \cdot \left(
	\begin{array}{cc}
	z& t\\
	\end{array}
	\right)$ et avec $x$ et $y$ premiers entre eux, et $z$ et $t$ premiers entre eux. \\
	La matrice $P$ est nécessairement de la forme $P = \left(
	\begin{array}{cc}
	u & x\\
	v & y
	\end{array}
	\right)$
	pour être dans $GL_2(\Z)$, puisque les coefficients $x$ et $y$ doivent être premiers entre eux.
	Déterminons toutes les valeurs possibles des entiers $u$ et $v$ pour que la matrice $P$ soit dans $\Gamma$.
	
	\begin{itemize}
		\item Si $x = 0$, alors on a n\'ecessairement $y = 1$, et donc $(u,v) = (1,0)$ convient et est la seule solution.
		\item Si $x = 1$, alors les solutions sont $(u,v) = (0,1)$ et $(u,v) = (1,y-1)$ (le couple $(u,v) = (1,y-1)$ est bien une solution \`a condition que $y \geq 2$).
		\item Si $x \geq 2$, alors les solutions $(u,v)$ v\'erifient n\'ecessairement les in\'egalit\'es $0 \leq u < x$ et $0 \leq v < y$. Le th\'eor\`eme de B\'ezout nous donne alors l'existence et l'unicit\'e d'une solution $(u,v)$ v\'erifiant les in\'egalit\'es $0 \leq u < x$ et $0 \leq v < y$ et l'\'equation $uy - vx = 1$, et de m\^eme pour l'\'equation $uy - vx = -1$. L'in\'egalit\'e $u \leq v$ est alors automatiquement bien v\'erifi\'ee. On a donc exactement deux solutions.
	\end{itemize}
	D'autre part, il est facile de v\'erifier que l'on a les égalités
	\[
		\m{n}\m{1} \left(
		\begin{array}{cc}
		0 &0\\
		0 &e
		\end{array}
		\right) = \m{n+1} \left(
		\begin{array}{cc}
		0 &0\\
		0 &e
		\end{array}
		\right)
			\text{ et que }
		\left(
		\begin{array}{cc}
		0 &0\\
		0 &e
		\end{array}
		\right) \m{1}\m{n} = \left(
		\begin{array}{cc}
		0 &0\\
		0 &e
		\end{array}
		\right) \m{n+1}
	\]
	d\`es que $n \geq 1$, d'o\`u le r\'esultat. Par transposition, on obtient également les matrices $Q$ qui conviennent.
\end{proof}


\section{Preuve du th\'eor\`eme~\ref{MN}} \label{preuve}

Dans cette partie, nous d\'emontrons le th\'eor\`eme~\ref{MN}.
Pour cela, nous commen\c{c}ons par le reformuler en termes de matrices :

\begin{thm} \label{MN2}
	Soit $A = MN$ avec $M$ et $N$ deux matrices sym\'etriques de $\Gamma_q$ telles que l'une d'elles est de d\'eterminant $-1$.
	Alors il existe des matrices $B$ et $C$ dans $\Gamma_{2q+1}$ telles que pour tout $n$, le corps de $BA^nC\transpA^n$ est le corps de $A$, et telles que la suite $BA^nC\transpA^n$ est non triviale.
\end{thm}

\begin{rem} \label{rem_MN2}
	Une matrice de la forme $MN$, avec $M, N \in \Gamma$ sym\'etriques et $\Det(M) = -1$ ou $\Det(N) = -1$, est toujours semblable (en faisant une permutation circulaire sur les entiers strictement positifs $m_1, ... , m_n$ qui apparaissent dans la d\'ecomposition $MN = \m{m_1} ... \m{m_n}$) \`a une matrice de la forme $\m{k}M'$, $k \geq 1$, $M' \in \Gamma$ sym\'etrique.
	L'\'enonc\'e du th\'eor\`eme~\ref{MN2} n'est donc pas plus g\'en\'eral que celui du th\'eor\`eme~\ref{MN}.
\end{rem}


Toute la suite du chapitre est consacr\'ee \`a la preuve du théorème~\ref{MN2}, qui est équivalent au théorème~\ref{MN}.

\begin{proof}[Preuve du théorème~\ref{MN2}]
	Soient $M$ et $N$ deux matrices de $\Gamma$, avec $\Det(M) = -1$ (on peut toujours se ramener \`a ce cas, quitte \`a tout transposer).
	Pour tout entier $k$, posons
	$$H_k := H_k(M,N) = MS_0+(MN)^{2k}, $$
	o\`u $S_0 = \left(
	\begin{array}{cc}
	0 &-1\\
	1 &0
	\end{array}
	\right)$.
	
	
	\begin{lemme} \label{tr-sym}
		Pour toute matrice sym\'etrique $P \in M_2(\R)$, on a $\Tr(S_0P) = 0$.
	\end{lemme}
	
	\begin{proof}
		On a $\Tr(S_0P) = \Tr(\transp{(S_0P)}) = \Tr(-PS_0) = -\Tr(S_0P)$.
	\end{proof}
	
	\begin{lemme} \label{Hk1}
		Pour tout entier $k$, $H_k$ est de rang $1$.
	\end{lemme}
	
	\begin{proof}
		D'après la formule (\ref{fm}) (page \pageref{fm}), on a
		\[ \Det(H_k) = \Det(MS_0) + \Det((MN)^{2k}) + \Tr(MS_0((MN)^{2k})^{\dagger}), \]
		et ici on a $((MN)^{2k})^{\dagger} = (MN)^{-2k}$ puisque $\det(MN) = \pm1$. \\
		Or on a d'une part $\Tr(MS_0(MN)^{-2k}) = \Tr(S_0(MN)^{-2k}M) = 0$ puisque $(MN)^{-2k}M$ est sym\'etrique,
		et d'autre part $\Det(MS_0) = -1$ puisque $\Det(M) = -1$. Et comme on a de plus $\Det((MN)^{2k}) = 1$, le d\'eterminant de la matrice $H_k$ est finalement nul. La non nullit\'e de $H_k$ est claire, puisque la matrice $-S_0$ n'est ni dans $\Gamma$ ni dans $\Gamma^{-1}$.
	\end{proof}
	
	\begin{lemme} \label{Sdet}
		Pour toute matrice $C \in M_2(\R)$, on a $CS_0\transp{C} = \Det(C)S_0$.
	\end{lemme}
	
	\begin{proof}
		V\'erification facile.
	\end{proof}
	
	\begin{lemme} \label{Hk3}
		Pour k assez grand, la matrice $H_k$ s'\'ecrit $$H_k = F \m{i} \left(
		\begin{array}{cc}
		0 &0\\
		0 &e
		\end{array}
		\right) \m{j} G$$ o\`u $i$ et $j$ sont des entiers sup\'erieurs ou \'egaux \`a 2, $F$ et $G$ sont des matrices de $\Gamma$ et $e$ est un entier sup\'erieur ou \'egal \`a $1$.
	\end{lemme}
	
	\begin{proof}
		Quitte \`a prendre $k$ sup\'erieur ou \'egal \`a 3, la matrice $N(MN)^{2k-1}$ a tous ses coefficients sup\'erieurs ou \'egaux \`a $3$ (le quatrième nombre de Fibonacci) qui est strictement sup\'erieur \`a $1$. Et donc si l'on pose $\left(
		\begin{array}{cc}
		a &b\\
		c &d
		\end{array}
		\right) = N(MN)^{2k-1}$, alors on a les in\'egalit\'es strictes $0 < a < b < d$ et $a < c < d$.
		Les m\^emes in\'egalit\'es mais larges sont alors satisfaites pour la matrice $S_0 + N(MN)^{2k-1}$, et elles le sont donc encore pour la matrice $H_k = M(S_0 + N(MN)^{2k-1})$ \\
		(i.e. si $\left(
		\begin{array}{cc}
		a' &b'\\
		c' &d'
		\end{array}
		\right) = H_k$, alors on a $0 \leq a' \leq b' \leq d'$ et $a' \leq c' \leq d'$). \\
		D'apr\`es la proposition~\ref{H->Gamma}, la matrice $H_k$ s'\'ecrit alors $H_k = P \left(
		\begin{array}{cc}
		0 &0\\
		0 &e
		\end{array}
		\right) Q$
		pour des matrices $P, Q$ dans $\Gamma$ et un entier $e \geq 1$.
		On a ensuite $$H_{k+2} = M(S_0 + N(MN)^{2k+3}) = MNMNM(S_0 + N(MN)^{2k-1})MNMN$$
		ce qui donne $H_{k+2} = MNMNH_kMNMN$.
		Les relations donn\'ees dans la proposition~\ref{H->Gamma} permettent alors d'obtenir $H_{k+2}$ de la forme annonc\'ee (c'est-\`a-dire avec $i \geq 2$ et $j \geq 2$), puisque $MNMN \in \Gamma\backslash \{ I_2, \m{1} \}$.
	\end{proof}
	
	On fixe $k$ assez grand et $F,G, i, j$ et $e$ pour avoir $H_k = F \m{i} \left(
	\begin{array}{cc}
	0 &0\\
	0 &e
	\end{array}
	\right) \m{j} G$ comme dans le lemme~\ref{Hk3}.
	Introduisons les matrices suivantes
		$$B = \left\{
	          \begin{array}{ll}
	            \transp{G}\m{(j-1, 1, e-1, j)}G & \quad \mathrm{si}\quad \Det(G) = 1 \\
	            \transp{G}\m{(j, e-1, 1, j-1)}G & \quad \mathrm{si}\quad \Det(G) = -1 \\
	          \end{array}
	        \right.$$
	        $$C = \left\{
	          \begin{array}{ll}
	            F\m{(i-1, 1, e-1, i)}\transp{F} & \quad \mathrm{si}\quad \Det(F) = 1 \\
	            F\m{(i, e-1, 1, i-1)}\transp{F} & \quad \mathrm{si}\quad \Det(F) = -1 \\
	          \end{array}
	        \right.$$
	 
	\begin{lemme} 
		\text{ } 
		\begin{enumerate}
			\item Les matrices $B$ et $C$ sont dans $\Gamma$.
			\item Chaque matrice $B$ et $C$ s'\'ecrit sous la forme $N - S_0$, o\`u $N$ est une matrice sym\'etrique de rang 1.
		\end{enumerate}
	\end{lemme} 
	
	\begin{proof}
		\begin{enumerate}
			\item Si $e > 1$, les matrices $B$ et $C$ sont clairement dans $\Gamma$, et si $e = 1$, l'\'egalit\'e $\m{(a, 0, b)} = \m{a+b}$, pour des entiers $a$ et $b$, montre que $B$ et $C$ sont encore dans $\Gamma$.
		
			\item Pour $k \geq 1$, la matrice $\m{(k-1, 1, e-1, k)} = \left(
			\begin{array}{cc}
			e &ek+1\\
			ek-1 &ek^2
			\end{array}
			\right)$ s'\'ecrit sous la forme $N - S_0$ avec $N$ sym\'etrique de rang 1.
			On a donc, pour une matrice $K$ de $\Gamma$, $\transp{K}\m{(k-1, 1, e-1, k)}K = \transp{K}NK - \Det(K)S_0$ d'apr\`es le lemme \ref{Sdet}, et la matrice $\transp{K}NK$ est encore sym\'etrique de rang 1. On obtient donc bien le r\'esultat.
		\end{enumerate}
	\end{proof}
	
	\begin{lemme} \label{ltrH}
		Pour toute matrice $A \in M_2(\R)$, on a les deux \'egalit\'es
		\[ \Tr(BAC\transpA) = (\Tr(H_kA))^2 - 2\Det(A), \]
		\[ \discr(BAC\transpA) = (\Tr(H_kA))^2((\Tr(H_kA))^2 - 4\Det(A)). \]
	\end{lemme}
	
	\begin{proof}
		On peut \'ecrire $B$ et $C$ sous la forme : $B = b_0\transp{b_0} - S_0$ et $C = c_0\transp{c_0} - S_0$ pour des vecteurs $b_0, c_0 \in M_{2,1}(\R)$ d\'efinis par $b_0 = \transp{G}\left(
			\begin{array}{ll}
				\sqrt{e} \\
				j\sqrt{e}
			\end{array}
			\right)$ et $c_0 = F\left(
			\begin{array}{ll}
				\sqrt{e} \\
				i\sqrt{e}
			\end{array}
			\right)$.
		On a alors $\Tr(BAC\transpA) = \Tr(b_0\transp{b_0}Ac_0\transp{c_0}\transpA) + \Tr(S_0AS_0\transpA)$ (les deux termes $\Tr(S_0Ac_0\transp{c_0}\transpA)$ et $\Tr(b_0\transp{b_0}AS_0\transpA)$ sont nuls puisque les matrices $Ac_0\transp{c_0}\transpA$ et $\transpA b_0\transp{b_0}A$ sont sym\'etriques).
		On a ensuite \\ $\Tr(b_0\transp{b_0}Ac_0\transp{c_0}\transpA) = (\transp{b_0}Ac_0)^2 = \Tr(c_0\transp{b_0}A)^2$, et $\Tr(S_0AS_0\transpA) = -2\Det(A)$.
		On v\'erifie que l'on a $c_0\transp{b_0} = H_k$, et on a alors obtenue la premi\`ere \'egalit\'e.
		
		La deuxi\`eme \'egalit\'e s'en d\'eduit alors facilement :
		\[
			\begin{aligned}
				\discr(BAC\transpA) &= \Tr^2(BAC\transpA) - 4\Det(BAC\transpA) \\
									&= \Tr^4(H_kA) - 4\Tr^2(H_kA)\Det(A) + (4 - 4) \Det^2(A) \\ 
									&=\Tr^2(H_kA)( \Tr^2(H_kA) - 4\Det(A) ) \\
			\end{aligned}
		\]
	\end{proof}
	
	\begin{lemme} \label{Hk2}
		Pour tous entiers $n$ et $k$, on a $\Tr(H_k(MN)^n) = \Tr((MN)^{n+2k})$.
	\end{lemme}
	
	\begin{proof}
		On a $\Tr(H_k(MN)^n) = \Tr(MS_0(MN)^n) + \Tr((MN)^{n+2k})$.
		Or, $\Tr(MS_0(MN)^n) = \Tr(S_0(MN)^nM) = 0$, parce que $(MN)^nM$ est sym\'etrique.
	\end{proof}
	
	Terminons maintenant la preuve du th\'eor\`eme~\ref{MN2}. On a
	\[
		\begin{aligned}
			\discr(B(MN)^nC(NM)^n) &= \Tr^2(H_k(MN)^n)(\Tr^2(H_k(MN)^n) - 4\Det(MN)^n) \\
								&=\Tr^2((MN)^{n+2k})( \Tr^2((MN)^{n+2k}) - 4\Det(MN)^{n}) \\ 
								&=\Tr^2((MN)^{n+2k})\discr((MN)^{n+2k}) \\
		\end{aligned}
	\]
	On v\'erifie que les matrices $B$ et $C$ sont bien dans $\Gamma_{2q+1}$ (voir sections~\ref{suites2} et \ref{exs} pour plus de d\'etails),
	et en explicitant les suites obtenues, on vérifie qu'elles sont non triviales (voir section~\ref{exs}).
	Ceci termine la preuve du th\'eor\`eme~\ref{MN2}.
\end{proof}

\section{Fractions continues de la forme $[ \overline{ \pat{B}\pat{A}^n } ]$} \label{suites1}

Dans cette partie, nous expliquons pourquoi la construction pr\'ec\'edente ne pouvait pas aboutir avec des suites de fractions continues de la forme $$ [ \overline{ b_1, b_2, ... , b_k (a_1, a_2, ..., a_l)^n } ]. $$
Nous montrons qu'il n'existe pas de suite non constante de fractions continues p\'eriodiques de la forme $ [ \overline{ \pat{B}\pat{A}^n } ] $ dans un corps quadratique donn\'e,
ce qui se formule matriciellement de la fa\c{c}on suivante :

\begin{prop} \label{prop-suites1}
	Soit $\Q[\sqrt{\delta}]$ un corps quadratique r\'eel, et soient $A$ et $B$ des matrices de $\Gamma$.
	Si pour tout $n$, le corps de $BA^n$ est $\Q[\sqrt{\delta}]$, alors la suite $BA^n$ est triviale (au sens de la définition \ref{def-triv}).
\end{prop}
On obtient même que les matrices $BA^n$ correspondent toutes \`a la m\^eme fraction continue p\'eriodique, puisque l'on va montrer qu'il existe une matrice $D \in M_2(\R)$ telle que pour tout $n$, $BA^n$ est une puissance de $D$.

Pour d\'emontrer cette proposition, nous aurons besoin de quelques lemmes : \\

\noindent Le lemme suivant permet de ramener le fait qu'une matrice soit dans un corps donn\'e \`a une \'egalit\'e de traces.

\begin{lemme} \label{pell}
	Soit $\delta$ entier non carr\'e.
	Alors il existe une matrice $U$ dans $GL_{2}(\Z)$ telle que les solutions positives $x$ \`a l'\'equation de Pell-Fermat : $x^2 - \delta y^2 = \pm4$ sont exactement les $\Tr(U^n), n \in \Z$.
\end{lemme}

\begin{proof}
	Ecrivons $\delta = k^2\delta'$ o\`u $\delta'$ est sans facteur carr\'e.
	
	Supposons d'abord que $k=1$.
	Montrons que dans ce cas les solutions $x$ à l'équation de Pell-Fermat $x^2 - \delta y^2 = \pm4$ sont exactement les traces des unités (c'est-à-dire des entiers de norme $\pm1$) du corps $\Q[\sqrt{\delta'}]$.
	
	Pour un réel quadratique $z = x' + y' \sqrt{\delta'}$, la trace vaut $\Tr(z) = 2x'$ et la norme vaut $\Norm(z) = {x'}^2 - \delta' {y'}^2$.
	On a donc
	\[ \Tr^2(z) - \delta' (2y')^2 = 4\Norm(z), \]
	et donc la trace est solution $x$ de l'équation $x^2 - \delta' y^2 = \pm4$ si et seulement si $N(z) = \pm1$. 
	On vérifie que si $N(z) = \pm1$, alors $z$ est un entier de $\Q[\sqrt{\delta'}]$ si et seulement si $\Tr(z)$ et $2y'$ sont des entiers, sachant que l'anneau des entiers est $\Z[\frac{\sqrt{\delta'}+1}{2}]$ si $\delta' \equiv 1$ (mod $4$) et est $\Z[\sqrt{\delta'}]$ si $\delta' \not\equiv 1$ (mod $4$).
%
%
	Ainsi, on obtient bien que les parties $x$ des solutions $(x,y)$ entières sont exactement les traces des unit\'es. 
	Or, le th\'eor\`eme des unit\'es de Dirichlet nous donne l'existence d'une unit\'e fondamentale $u \in O_{\delta'}^*$ dont les puissances engendrent, au signe pr\`es, le groupe des unit\'es.
	
	Soit $U$ la matrice de la multiplication par $u$ dans la base $\{1, u\}$. Alors $U \in GL_{2}(\Z)$, puisque $u$ est dans l'anneau d'entiers et son inverse aussi, et on a pour tout $n$, $\Tr(u^{n}) = \Tr(U^{n})$, d'o\`u le r\'esultat.
	
	Si maintenant on ne fait plus d'hypoth\`ese sur $k$, on remarque qu'un couple $(x,y)$ est solution de l'\'equation $x^2 - \delta y^2 = \pm 4$ si et seulement si $(x, ky)$ est solution de $x^2 - \delta' y^2 = \pm 4$ avec $x$ et $y$ entiers. Or, l'ensemble des unités $z = x' + y' \sqrt{\delta'}$ telles que $2y'$ est divisible par $k$ est un sous-groupe du groupe des unités. Donc il existe une certaine puissance de la matrice $U$ obtenue pr\'ec\'edemment qui convient.
\end{proof}

Le lemme suivant permettra de montrer (entre autres) que si les corps des matrices $BA^n$ sont tous les m\^emes, alors la matrice $A$ a aussi le m\^eme corps.

\begin{lemme} \label{infini-toujours}
	Soit $n_0 \in \Z$, soient $A, B, C$ trois matrices de $M_2(\R)$, avec $A$ et $C$ ayant chacune des valeurs propres distinctes en module, et avec $C$ inversible, et soient $a$ et $b$ deux entiers.
	Si l'\'egalit\'e $\Tr(BA^n) = \Tr(C^{an+b})$ est vraie pour tout $n \geq n_0$, alors les matrices $A$ et $C^a$ ont m\^emes valeurs propres, et l'\'egalit\'e a lieu pour tout $n \in \Z$.
\end{lemme}

\begin{proof}
	Si l'on appelle $\lambda$ et $\overline{\lambda}$ les valeurs propres de $A$, avec $\abs{\overline{\lambda}} < \abs{\lambda}$, et  $\mu$ et $\overline{\mu}$ les valeurs propres de $C$, avec $\abs{\overline{\mu}} < \abs{\mu}$, l'\'egalit\'e des traces se r\'e\'ecrit :
	\( \alpha \lambda^n + \beta \overline{\lambda}^n = \mu^{an+b} + \overline{\mu}^{an+b} \) pour certains r\'eels $\alpha$ et $\beta$ et pour tout entier $n \geq n_0$.
	
	Regardons les termes dominants de part et d'autre.
	Si on avait $\alpha = 0$, on aurait $\overline{\lambda} = \mu^a$ et $\beta = \mu^b$, puis $\overline{\mu} = 0$, ce qui contredit l'inversibilit\'e de $C$.
	On a donc $\alpha \neq 0$.
	On obtient donc $\lambda = \mu^a$ et $\alpha = \mu^b$, puis on obtient $\overline{\lambda} = \overline{\mu}^a$ et $\beta = \overline{\mu}^b$.
	Donc $A$ et $C$ ont m\^emes valeurs propres et l'\'egalit\'e a lieu pour tout $n \in \Z$.
\end{proof}

D'apr\`es le lemme qui suit, les matrices $BA^n$ ont toutes pour corps $\Q[\sqrt{\delta}]$ d\`es qu'il existe deux entiers $i$ et $j$ tels que les matrices $BA^i$ et $BA^j$ ont pour corps $\Q[\sqrt{\delta}]$.

\begin{lemme} \label{2-toujours}
	Soient $A$ et $C$ des matrices positives ayant m\^emes valeurs propres, et $B$ et $D$ des matrices de $M_2(\R)$ quelconques. 
	Si la relation $\Tr(BA^n) = \Tr(DC^n)$ est vraie pour deux valeurs de $n$, alors elle est vraie pour tout $n \in \Z$.
\end{lemme}

\begin{proof} 
	Soient $\lambda$ et $\overline{\lambda}$ les deux valeurs propres distinctes de $A$ (et donc aussi de $C$).
	La relation $\Tr(BA^n) = \Tr(DC^n)$ se r\'e\'ecrit $\alpha \lambda^n + \beta \overline{\lambda}^n = 0$ pour des r\'eels $\alpha$ et $\beta$ indépendants de $n$, 
	puisque les coefficients de chacune des deux matrices $BA^n$ et $DC^n$ sont des combinaisons linéaires en $\lambda^n$ et $\overline{\lambda}^n$.
	Si la relation est vraie pour deux valeurs de $n$, on obtient alors un syst\`eme de Cramer, donc $\alpha = \beta = 0$.
\end{proof}


La suite de ce chapitre est consacr\'ee \`a la preuve de la proposition~\ref{prop-suites1}.

\begin{proof}[Preuve de la proposition~\ref{prop-suites1}]
	Supposons que l'on ait une suite de matrices $BA^n$ ayant toutes pour corps $\Q[\sqrt{\delta}]$, o\`u $A$ et $B$ sont deux matrices positives et $\delta$ est un entier sans facteur carr\'e.
	Il existe alors un entier $\alpha_n$ tel que le discriminant $\discr(BA^n) = \Tr(BA^n)^2 - 4\Det(BA^n)$ soit \'egal \`a $\delta \alpha_n^2$.
	Et donc $\Tr(BA^n)$ est solution $x$ de l'\'equation de Pell-Fermat $x^2 - \delta y^2 = 4\Det(BA^n)$. D'apr\`es le lemme~~\ref{pell}, il existe donc une matrice $U \in GL_{2}(\Z)$ (ind\'ependante de $n$) telle que pour tout $n$, il existe un entier $i_n$ tel que $\Tr(BA^n) = \Tr(U^{i_n})$, et on a alors aussi $\Det(BA^n) = \Det(U^{i_n})$ \`a partir d'un certain rang. On peut alors utiliser le lemme suivant :
	
	\begin{lemme} \label{arithm}
		Soient $U$ une matrice de $GL_2(\Z)$ ayant des valeurs propres distinctes en module, $B \in M_2(\R)$ une matrice, et $A$ une matrice positive.
		Si $(i_n)$ est une suite d'entiers telles que pour tout $n$, $\Tr(BA^n) = \Tr(U^{i_n}) \neq 0$,
		alors, la suite $(i_n)$ est arithm\'etique.
	\end{lemme}
	
	\begin{proof}
		Si l'on appelle $\lambda$ et $\mu$ les deux valeurs propres de la matrice positive $A$ (avec $\lambda > 1$ et $\abs{\mu} < 1$), alors la trace de $BA^n$ s'\'ecrit $ e\lambda^n + f\mu^n$ pour certains r\'eels $e$ et $f$ indépendants de $n$, 
		et on a $\Tr(U^{i_n}) = \alpha^{i_n} + \beta^{i_n}$, o\`u $\alpha$ et $\beta$ sont les deux valeurs propres de la matrice $U$ (avec $\alpha$ la plus grande valeur propre de $U$ en module).
		En divisant par $\alpha^{i_n}$ de part et d'autre de l'\'egalit\'e $e\lambda^n + f\mu^n = \alpha^{i_n} + \beta^{i_n}$, on obtient $\lim_{n \rightarrow \infty}{\cfrac{e\lambda^n}{\alpha^{i_n}}} = 1$.
		En prenant le $\log$, on obtient alors que $\lim_{n \rightarrow \infty}{i_n\ \log(\alpha) - n\ \log(\lambda) - \log(e)} = 0$.
		On a donc $i_n = an + b + \epsilon_n$, o\`u $\lim_{n \rightarrow \infty}{\epsilon_n} = 0$, $a = \cfrac{\ \log(\lambda)}{\ \log(\alpha)}$ et $b = \cfrac{\ \log(e)}{\ \log(\alpha)}$.
		Et comme $i_n$ est entier, on a finalement $i_n = an + b$ \`a partir d'un certain rang. D'apr\`es le lemme~~\ref{infini-toujours}, cela est alors vrai pour tout $n$, ce qui termine la preuve du lemme.
	\end{proof}
	
	Dans ce dernier lemme $a$ et $b$ sont des entiers, puisque $i_0$ et $i_1$ sont entiers, et le lemme~~\ref{infini-toujours} donne que les matrices $A$ et $U^a$ ont m\^emes valeurs propres.
	On obtient donc les \'egalit\'es $\Tr(BA^n) = \Tr(U^{an+b})$ et $\Det(BA^n) = \Det(U^{an+b})$ pour tout $n$.
	L'\'egalit\'e des traces nous donne que dans une base dans laquelle $A$ est diagonale, la matrice $B$ est de la forme $\left(
	\begin{array}{cc}
		\alpha^b &*\\
		* &\beta^b
	\end{array}
	\right)$ o\`u $\alpha$ et $\beta$ sont les valeurs propres de $U$.
	L'\'egalit\'e des d\'eterminants $\Det(B) = \Det(U^b) = \alpha^b \beta^b$ implique alors que $B$ est trigonale dans la base de diagonalisation de A.
	Les matrices $A$ et $B$ sont \`a coefficients entiers et ont un espace propre en commun : elles sont donc simultan\'ement diagonalisables, puisque l'on obtient le deuxi\`eme espace propre par l'\'element non trivial de $Gal(\Q[\sqrt{\delta}]/\Q)$.
	Si l'on pose $D$ la matrice qui vaut $\left(
	\begin{array}{cc}
		\alpha &0\\
		0 &\beta
	\end{array}
	\right)$ dans la base de diagonalisation de $A$, on a $B = D^b$ et $A = D^a$, donc la suite $BA^n$ est triviale.
	Ceci termine la preuve de la proposition \ref{prop-suites1}. 
\end{proof}

\section{Fractions continues de la forme $ [ \overline{ \pat{B}\pat{A}\pat{C}\pattranspA } ] $} \label{BACtA}

Dans ce chapitre, nous \'etudions les fractions continues p\'eriodiques correspondant aux matrices de la forme $BAC\transpA$. Exp\'erimentalement, de telles fractions continues p\'eriodiques apparaissent souvent, et nous allons tenter d'expliquer pourquoi, et en m\^eme temps g\'en\'eraliser et donner les r\'eciproques de r\'esultats qui permettent d'aboutir au th\'eor\`eme~\ref{MN}.

Nous allons voir dans ce chapitre que l'on peut ramener l'\'etude des suites de matrices de la forme $BA^nC\transpA^n$ (qui est faite dans le chapitre~\ref{suites2}) \`a l'\'etude de suites de la forme $HA^n$ pour certaines matrices $H$ non inversibles. On se ram\`ene donc \`a des suites d'une forme semblable \`a celles qui ont \'etudi\'ees dans le chapitre pr\'ec\'edent.

\begin{lemme} \label{lB+tB}
	Soit $B$ une matrice de $GL_2(\Z)$.
	On a équivalence entre les deux points suivants :
	\begin{enumerate}
		\item	$B+\transp{B}$ est de rang 1,
		\item Il existe une matrice $N$ sym\'etrique de rang 1 telle que $B = N \pm S_0$,
	\end{enumerate}
	où $S_0$ est la matrice définie dans le chapitre \ref{preuve}. \\
	De plus, si l'un de ces points est satisfait, alors on a $\Det(B) = 1$.
\end{lemme}

\begin{proof}
	\underline{$\Longrightarrow$}
	Si $B = \left(
	\begin{array}{cc}
	a &b\\
	c &d
	\end{array}
	\right)$, on a l'\'egalit\'e
	$$0 = \Det(B+\transp{B}) = 4ad-(b+c)^2 = 4\Det(B) - (b-c)^2.$$
	Donc on obtient $\Det(B) = 1$ et $b-c = \pm 2$.
	Quitte \`a la transposer, la matrice $B$ peut donc s'\'ecrire 
	$B = \left(
	\begin{array}{cc}
	a &b'-1\\
	b'+1 &d
	\end{array}
	\right)$, avec $ad = b'^2$.
	
	\noindent \underline{$\Longleftarrow$}
	Clair.
\end{proof}

La proposition qui suit g\'en\'eralise le lemme~\ref{ltrH}.

\begin{prop} \label{FQ}
	Soient $B$ et $C$ deux matrices de $GL_2(\Z)$.
	On a l'\'equivalence :
	\begin{enumerate}
		\item  Il existe une matrice $H \in \M_2(\R)$ de rang 1 et un r\'eel $\lambda$ tels que pour toute matrice $A \in M_2(\R)$ on ait :
		$\Tr(BAC\transpA) = \Tr^2(HA) + \lambda \Det(A)$.
		\item Les matrices $B+\transp{B}$ et $C+\transp{C}$ sont de rang $1$.
	\end{enumerate}
	De plus, si l'un des points est satisfait, alors on a n\'ecessairement $\lambda = \pm 2$, et on a l'\'egalit\'e suivante pour toute matrice $A \in M_2(\R)$ :
	\[ \discr(BAC\transp{A}) = \Tr^2(HA)(\Tr^2(HA) \pm 4\Det(A)). \]
\end{prop}

Tous les exemples connus de suites de fractions continues p\'eriodiques qui restent dans un corps quadratique donn\'e ( voir par exemple \cite{mcmullen}, \cite{wilson} et le chapitre~\ref{exs} ci-apr\`es ), correspondent \`a des suites de matrices de la forme $BA^nC\transpA^n$ avec $B+\transp{B}$ et $C+\transp{C}$ de rang $1$. Nous ignorons si cela est toujours vrai.

L'expression du discriminant $\discr(BAC\transp{A})$ donn\'ee par la proposition \ref{FQ}, donne une factorisation par un carr\'e. Cela est favorable \`a ce que le corps quadratique de la matrice $BAC\transpA$ soit petit et explique donc un peu pourquoi l'on observe un certain nombre de fractions continues p\'eriodiques de cette forme.

\begin{proof}[Preuve de la proposition~\ref{FQ}]
	\underline{$\Longrightarrow$}
	Ecrivons les matrices $B$ et $C$ sous la forme : $B = B_0 + \beta S_0$ et $C = C_0 + \gamma S_0$, o\`u $B_0$ et $C_0$ sont deux matrices sym\'etriques de $M_2(\R)$, $\beta$ et $\gamma$ sont deux r\'eels et $S_0$ est la matrice $\left( \begin{array}{cc} 0 & -1 \\ 1 & 0 \end{array} \right)$.
	On a alors $\Tr(S_0AC_0\transp{A}) = \Tr(B_0AS_0\transp{A}) = 0$ puisque les matrices $AC_0\transp{A}$ et $\transp{A}B_0A$ sont sym\'etriques, et on a $\Tr(S_0AS_0\transp{A}) = -2\Det(A)$.
	On obtient donc l'\'egalit\'e $\Tr(BAC\transp{A}) = \Tr(B_0AC_0\transp{A}) - 2\beta\gamma \Det(A)$.
	On souhaite maintenant montrer que les matrices $B_0$ et $C_0$ sont chacune de rang 1.
	En \'evaluant la forme quadratique $A \mapsto \Tr(BAC\transp{A})$ en $A = \left( \begin{array}{cc} x & y \\ 0 & 0 \end{array} \right)$, en $A = \left( \begin{array}{cc} 0 & 0 \\ x & y \end{array} \right)$, et en les transpos\'ees, on obtient \`a chaque fois des formes quadratiques en $x$ et $y$ qui doivent \^etre des carr\'es, et qui ont \`a chaque fois pour matrice un multiple de $B_0$ ou de $C_0$.
	Ceci nous donne que les matrices $B_0$ et $C_0$ sont chacune des matrices de formes quadratiques carr\'ees.
	Et ni la matrice $B_0$ ni la matrice $C_0$ ne peuvent \^etre nulles puisque la matrice $H$ est non nulle. Donc les matrices $B_0$ et $C_0$ sont de rang $1$.
	
	\underline{$\Longleftarrow$}
	Si des matrices $B$ et $C$ de $GL_2(\Z)$ sont telles que $B+\transp{B}$ et $C+\transp{C}$ sont de rang $1$, d'apr\`es le lemme~\ref{lB+tB} on peut les \'ecrire sous la forme $B = b_0\transp{b_0} \pm S_0$ et $C = c_0\transp{c_0} \pm S_0$ pour des vecteurs $b_0$ et $c_0$ de $M_{2,1}(\R)$.
	On a alors $\Tr(BAC\transpA) = \Tr(b_0\transp{b_0}Ac_0\transp{c_0}\transpA) \pm \Tr(S_0AS_0\transpA)$, puisque les deux termes $\Tr(S_0Ac_0\transp{c_0}\transpA)$ et $\Tr(b_0\transp{b_0}AS_0\transpA)$ sont nuls, \'etant donn\'e que les matrices $Ac_0\transp{c_0}\transpA$ et $\transpA b_0\transp{b_0}A$ sont sym\'etriques.
	On a ensuite $\Tr(b_0\transp{b_0}Ac_0\transp{c_0}\transpA) = (\transp{b_0}Ac_0)^2 = \Tr(c_0\transp{b_0}A)^2$, et $\Tr(S_0AS_0\transpA) = -2\Det(A)$, d'o\`u le r\'esultat avec $H := c_0\transp{b_0}$. \\
	L'expression du discriminant annonc\'ee se d\'eduit facilement du point 1. :
	\begin{align*}
		\discr(BAC\transpA) &= \Tr^2(BAC\transpA) - 4\Det(BAC\transpA) \\
							&= \Tr^4(HA) \pm 4\Tr^2(HA)\Det(A) + 4 - 4\Det(BC) \\ 
							&=\Tr^2(HA)( \Tr^2(HA) \pm 4\Det(A) ).
	\end{align*}
\end{proof}

\begin{rem} \label{rqFQ}
	Dans la proposition~\ref{FQ}, si $B$ est de la forme $b_0\transp{b_0} + \epsilon_B S_0$ et $C$ est de la forme $c_0\transp{c_0}+\epsilon_C S_0$ pour des vecteurs $b_0$ et $c_0$ de $M_{2,1}(\R)$ et des signes $\epsilon_B$ et $\epsilon_C$ de $\{ -1, 1 \}$,
	alors $\lambda$ vaut $-2\epsilon_B\epsilon_C$ et $H$ vaut (au signe pr\`es) $c_0\transp{b_0}$ et ce sont les seules solutions.
\end{rem}

La proposition suivante permet de d\'eterminer toutes les matrices $B \in \Gamma$ telles que $B+\transp{B}$ est de rang 1.

\begin{prop} \label{B+tB}
	Soit $B \in \Gamma$. On a l'\'equivalence :
	\begin{enumerate}
		\item	$B+\transp{B}$ est de rang 1,
		\item	Il existe des entiers $k \geq 1$ et $n \geq 2$, et une matrice $F$ de $\Gamma$ tels que $B$ ou $\transp{B}$ vaut $F\m{(n-1, 1, k-1, n)}\transp{F}$.
	\end{enumerate}
\end{prop}
La transposée d'une matrice positive est positive, et on a $\m{(n, 1, 0, n+1)} = \m{(n, n+2)}$, donc le deuxi\`eme point de la proposition~\ref{B+tB} entra\^ine automatiquement que $B$ est une matrice positive.

\begin{proof} 
	\underline{$\Longleftarrow$}
	V\'erification facile sachant que l'on a
	$$\m{(n-1, 1, k-1, n)} = \left(
	\begin{array}{cc}
	k &kn+1\\
	kn-1 &kn^2
	\end{array}
	\right).$$
	
	\noindent \underline{$\Longrightarrow$}
	Quitte \`a transposer la matrice $B$, on peut l'\'ecrire sous la forme : $\left(
	\begin{array}{ccc}
	a & b+1\\
	b-1 & c
	\end{array}
	\right)$ avec $ac = b^2$.
	On peut alors \'ecrire les entiers $a$ et $c$ sous la forme : $a = zx^2$ et $c = z'y^2$ avec $z$ et $z'$ sans facteurs carr\'es. La condition $ac = b^2$ entra\^ine alors que $z = z'$ et $b = \pm xyz$.
	Comme la matrice $B$ est dans $\Gamma$, on peut donc \'ecrire 
	$$B = \left(
	\begin{array}{ccc}
	zx^2 &xyz+1\\
	xyz-1 &zy^2
	\end{array}
	\right), \quad \text{avec} \quad x,y,z \geq 1.$$
	On peut supposer que la premi\`ere matrice $\m{i}$ et la derni\`ere matrice $\m{j}$ qui apparaissent dans la d\'ecomposition de $B$ en produit de matrices $\m{k}$, $k \geq 1$, sont distinctes. En effet, ni l'identit\'e $I_2$ ni les matrices $\m{i}$, $i \geq 1$ ne sont de la forme $N \pm S_0$, avec $N$ matrice sym\'etrique de rang 1. \\
	On a ensuite $i = \floor{ \cfrac{xyz-1}{zx^2} }$ et $j = \floor{ \cfrac{xyz+1}{zx^2} }$, puisque $\Det(B) = 1 > 0$ et $zx^2 > 0$.
	De plus, on a $i < j$ puisque l'on a $i \neq j$. Et on a les in\'egalit\'es :
	$$izx^2 \leq xyz-1 < (i+1)zx^2,$$
	$$jzx^2 \leq xyz+1 < (j+1)zx^2.$$
	Donc en particulier on a $(j-i-1)zx^2 < 2$.
	On obtient alors deux cas :
	
	\emph{Premier cas} : $x = z = 1$ \\
	On a alors $B = \left(
	\begin{array}{ccc}
	1 &y+1\\
	y-1 &y^2
	\end{array}
	\right) = \m{(y-1, y+1)} = \m{(y-1, 1, 0, y)}$.
	
	\emph{Deuxi\`eme cas} : $xz \geq 2$ et j = i+1\\
	On a $y - \cfrac{1}{xz} < jx \leq y+ \cfrac{1}{xz}$, avec $xz \geq 2$ et comme $y$ et $jx$ sont des entiers, ceci entra\^ine que $y = jx$.
	Et donc finalement
	$$ B = \left(
		\begin{array}{ccc}
			zx^2 &jzx^2+1\\
			jzx^2-1 &zj^2x^2
		\end{array}
		\right) \\
		 =  \m{(j-1, 1, zx^2-1, j)}. $$
\end{proof}

Le corollaire qui suit nous permet de conna\^itre la matrice $H$ qui appara\^it dans l'\'ecriture $\Tr(BAC\transpA) = \Tr^2(HA) \pm 2\Det(A)$ pour toute matrice $A \in M_2(\R)$, en fonction des d\'ecompositions des matrices positives $B$ et $C$ comme produits de matrices $\m{k}$, donn\'ees par la proposition pr\'ec\'edente.

Etant fix\'ee une matrice $A$ positive de corps $\Q[\sqrt{\delta}]$, cela nous permet de ramener la recherche de matrices $B$ et $C$ telles que les rangs de $B+\transp{B}$ et de $C+\transp{C}$ sont 1, et telles que les corps des matrices $BA^nC\transpA^n$ sont tous les m\^emes, \`a la recherche des matrices $H$ de rang 1 telles que $\Tr^2(HA)(\Tr^2(HA) \pm 4\Det(A))$ ( c'est-\`a-dire le discriminant de $BAC\transpA$ ) soit de la forme $\delta\alpha^2$ pour $\alpha$ entier. 

\begin{cor} \label{H}
	Si l'on a
	$$
	\begin{array}{ccl}
		B & = & M\m{(m-1, 1, k-1, m)} \transp{M}, \\
		C & = & N\m{(n-1, 1, l-1, n)} \transp{N},
	\end{array}
	$$
	pour $n,m \geq 1$, $k,l \geq 0$ et $M, N \in GL_2(\Z)$, alors la matrice $H \in M_2(\R)$ telle que $\Tr(BAC\transpA) = (\Tr(HA))^2 \pm 2\Det(A)$ pour toute matrice $A \in M_2(\R)$, s'\'ecrit :
	$$H = N\m{n} \left(
	\begin{array}{cc}
	0 &0\\
	0 &\sqrt{kl}
	\end{array}
	\right) \m{m}\transp{M}.$$
\end{cor}

\begin{proof} 
	On peut se ramener \`a $M = N = I_2$, puisque si \mbox{$B = MB'\transp{M}$} et \mbox{$C = NC'\transp{N}$}, alors $\Tr(BAC\transpA) = \Tr(B'(\transp{M}AN)C'\transp{(\transp{M}AN)})$, et $A \mapsto \transp{M}AN$ d\'ecrit $M_2(\R)$ quand $A$ d\'ecrit $M_2(\R)$ puisque $M$ et $N$ sont inversibles. \\
	La proposition~\ref{FQ} donne bien l'existence d'une matrice $H \in M_2(\R)$ telle que $\Tr(BAC\transpA) = (\Tr(HA))^2 \pm 2\Det(A)$ pour toute matrice $A \in M_2(\R)$, et d'apr\`es la remarque~\ref{rqFQ} elle est uniquement d\'etermin\'ee au signe pr\`es, et vaut
	$H = \left(
	\begin{array}{cc}
	\sqrt{ae} &\sqrt{ah}\\
	\sqrt{ed} &\sqrt{hd}
	\end{array}
	\right)$ si $B = \left(
	\begin{array}{cc}
	a &*\\
	* &d
	\end{array}
	\right)$ et $C = \left(
	\begin{array}{cc}
	e &*\\
	* &h
	\end{array}
	\right)$. 
	Or, $B$ vaut $\left(
	\begin{array}{cc}
	k & mk+1\\
	mk-1 & km^2
	\end{array}
	\right)$, et $C$ vaut $\left(
	\begin{array}{cc}
	l & nl+1\\
	nl-1 & ln^2
	\end{array}
	\right)$, d'o\`u le r\'esultat.
\end{proof}

\begin{cor}
	Soit $\delta$ un entier sans facteur carr\'e, soit $A$ une matrice positive, soient $B$ et $C$ des matrices de $\Gamma$ telles que $B+\transp{B}$ et $C+\transp{C}$ sont de rang $1$, et soit $H$ la matrice donn\'ee dans la proposition pr\'ec\'edente.\\
	On a l'\'equivalence:
	\begin{enumerate}
		\item Le corps de $BAC\transpA$ ou de $BA\transp{C}\transpA $ est $\Q[\sqrt{\delta}]$.
		\item L'entier $\cfrac{1}{\sqrt{s}} \Tr(HA)$ est solution $x$ d'une \'equation de Pell-Fermat $s x^2 - t y^2 = \pm 4$, pour des entiers $s$ et $t$ tels que $\cfrac{1}{\sqrt{s}} H \in M_2(\Z)$ et tels que le produit $st$ soit de la forme $\delta k^2$.
	\end{enumerate}
	En particulier, il suffit que $\Tr(HA)$ soit solution $x$ enti\`ere de l'\'equation $x^2 - \delta y^2 = \pm 4$ pour que le corps de $BAC\transpA$ ou de $BA\transp{C}\transpA $ soit $\Q[\sqrt{\delta}]$.
\end{cor}

\begin{rem}
	Il est facile de d\'eterminer dans ce corollaire si c'est $BAC\transpA$ ou bien $BA\transp{C}\transpA $ dont le corps est $\Q[\sqrt{\delta}]$.
	Supposons que $B$ et $C$ s'\'ecrivent respectivement $N_B+\epsilon_BS_0$ et $N_C+\epsilon_CS_0$ pour des matrices $N_B$ et $N_C$ sym\'etriques de rang $1$.
	\begin{itemize}
		\item	Si l'entier $\cfrac{1}{\sqrt{s}} \Tr(HA)$ est solution $x$ de l'\'equation de Pell-Fermat $s x^2 - t y^2 = 4\epsilon_B\epsilon_C\Det(A)$, alors la matrice $BAC\transpA$ a pour corps $\Q[\sqrt{\delta}]$.
		\item	Si l'entier $\cfrac{1}{\sqrt{s}} \Tr(HA)$ est solution $x$ de l'\'equation de Pell-Fermat $s x^2 - t y^2 = -4\epsilon_B\epsilon_C\Det(A)$, alors la matrice $BA\transp{C}\transpA$ a pour corps $\Q[\sqrt{\delta}]$.
	\end{itemize}
\end{rem}

Dans la suite, nous nous int\'eressons surtout au cas particulier o\`u $\Tr(HA)$ est solution $x$ enti\`ere de l'\'equation $x^2 - \delta y^2 = \pm 4$. 
Le cas g\'en\'eral est plus compliqu\'e, car les ensembles de solutions des \'equations de Pell-Fermat plus g\'en\'erales $s x^2 - t y^2 = \pm 4$ n'ont pas une structure aussi simple que pour l'\'equation classique o\`u $s=1$.
On retombe quand m\^eme sur une \'equation de Pell-Fermat classique dans le cas o\`u $s = \delta$ (voir les suites de type Wilson \`a la fin du chapitre~\ref{suites2}).

\begin{proof}[Preuve du corollaire]
	On a pour toute matrice $A \in M_2(\R)$, $\Tr(BAC\transpA) = \Tr(HA)^2 \pm 4\Det(A)$, et on peut choisir le signe devant le terme $4\Det(A)$ quitte \`a transposer $C$.
	On a alors $\discr(BAC\transpA) = \Tr^2(HA)(\Tr^2(HA) \pm 4\Det(A)) = s x^2 (s x^2 \pm 4)$ o\`u $x = \cfrac{1}{\sqrt{s}} \Tr(HA)$.
	Le corps de $BAC\transpA$ est donc $\Q[\sqrt{\delta}]$ si et seulement si $s x^2 (s x^2 \pm 4)$ est de la forme $\delta\alpha^2$, si et seulement si $s x^2 - t \alpha^2 = \mp 4$ pour un entier $t$ tel que $st$ est de la forme $\delta k^2$.
\end{proof}


\begin{conj} \label{maconj}
	Dans tout corps quadratique r\'eel, il existe une infinit\'e de fractions continues p\'eriodiques de la forme $ [ \overline{ 2,1,1,1,\pat{A},2,1,1,1,\pattranspA } ] $ ou de la forme  $ [ \overline{ 2,1,1,1,\pat{A},1,1,1,2,\pattranspA } ] $ form\'ees seulement des entiers 1 et 2.
\end{conj}
Ici, $\pat{A}$ est un motif $a_1, a_2, ... a_k$, avec $a_i \in \{1,2\}$, et $\pattranspA$ est le motif miroir. 

\begin{rem}
	Pour obtenir cette conjecture, il suffirait, \'etant donn\'e un entier $\delta$ non carr\'e, de trouver une infinit\'e de matrices A dans $\Gamma_2$ telles que $\Tr(HA)$ soit solution $x$ enti\`ere d'une \'equation de Pell-Fermat $x^2-\delta y^2 = \pm 4$, o\`u $H$ est la matrice $\left(
	\begin{array}{cc}
	2 &4\\
	4 &8
	\end{array}
	\right)$.
	
	La conjecture~\ref{conjMcMullen} de McMullen revient \`a trouver, pour tout $\delta$ non carr\'e, une infinit\'e de matrices $A$ dans $\Gamma_2$ telles que $\Tr(A)$ est solution $x$ de l'\'equation de Pell-Fermat $x^2-\delta y^2 = 4\Det(A)$.
	
	La conjecture~\ref{maconj} se ram\`ene donc approximativement \`a la conjecture~\ref{conjMcMullen} dans laquelle on aurait remplac\'e la $trace$ par la forme lin\'eaire $A \mapsto \Tr(HA)$. \\
	
	Il est malheureusement impossible de recommencer tel quel ce proc\'ed\'e qui nous a permis de passer de la trace \`a la forme lin\'eaire $A \mapsto \Tr(HA)$ : Si l'on a $\Tr(HBAC \transpA D) = (g(A))^2 + \lambda \Det(A)$ pour toute matrice $A \in M_2(\R)$, alors on a $\lambda = 0$, quelles que soient les matrices $B, C, D \in GL_2(\Z)$ et $g \in M_2(\R)^*$, et donc cela ne donne plus de factorisation du discriminant par un carr\'e.
\end{rem}

\begin{rem}
	La conjecture \ref{maconj} est \og presque\fg\ conséquence de celle de Zaremba avec une borne $2$. Voir section \ref{s_zaremba} pour plus de détails.
\end{rem}

\section{Fractions continues de la forme $[ \overline{ \pat{B}\pat{A}^n\pat{C}\pattranspA^n } ]$} \label{suites2}



Dans ce chapitre, nous donnons une fa\c{c}on de construire des suites de fractions continues p\'eriodiques qui sont dans un corps quadratique donn\'e, et qui correspondent \`a des suites de matrices de la forme $BA^nC\transpA^n$ avec $B+\transp{B}$ et $C+\transp{C}$ de rang $1$. 


\subsection{Hypoth\`ese H enti\`ere}
On supposera que $B$ et $C$ sont deux matrices de $\Gamma$ telles que $B+\transp{B}$ et $C+\transp{C}$ sont de rang $1$.
Et on supposera que la matrice $H$ donn\'ee par le corollaire~\ref{H} (pour ces matrices $B$ et $C$) est \`a coefficients entiers. \\

La proposition suivante justifie que pour une matrice positive $A$ fix\'ee on cherche \`a obtenir une relation de la forme $\Tr(HA^n) = \Tr(A^{n+k})$ pour obtenir une suite de matrices $BA^nC\transpA^n$ qui ont toutes le m\^eme corps. Dans le cas o\`u $A$ est une matrice donnée par le lemme \ref{pell}, et sous l'hypoth\`ese $H$ enti\`ere, cela est n\'ecessaire.

\begin{prop} \label{Hunit}
	Soit $\Q[\sqrt{\delta}]$ un corps quadratique r\'eel.
	Sous l'hypoth\`ese $H$ enti\`ere les deux assertions suivantes sont \'equivalentes:
	\begin{enumerate}
		\item	Pour tout $n \in \Z$, le corps de $BA^nC\transpA^n$ ou bien de $BA^n\transp{C}\transpA^n$ est $\Q[\sqrt{\delta}]$, \label{Hunit1}
		\item Il existe une matrice $U \in GL_2(\Z)$ de corps $\Q[\sqrt{\delta}]$ et deux entiers $a$ et $b$ tels que $\Tr(HA^n) = \Tr(U^{an+b})$ pour tout entier $n \in \Z$. \label{Hunit2}
	\end{enumerate}
	Et le corps de $A$ est alors n\'ecessairement $\Q[\sqrt{\delta}]$.
\end{prop}

\begin{rem} \label{rqHunit}
Supposons que l'on ait $B = N_B+\epsilon_BS_0$ et $C = N_C+\epsilon_CS_0$, pour des matrices $N_B$ et $N_C$ sym\'etriques de rang $1$, et pour des $\epsilon_B, \epsilon_C \in \{ -1, 1\}$.
Si l'on a $\Tr(HA^n) = \Tr(U^{an+b})$ pour tout entier $n \in \Z$, comme dans la proposition~\ref{Hunit}, alors
\begin{itemize}
	\item	la matrice $BA^nC\transpA^n$ a pour corps $\Q[\sqrt{\delta}]$ \quad si $\epsilon_B\epsilon_C = \Det(U^b)$,
	\item la matrice $BA^n\transp{C}\transpA^n$ a pour corps $\Q[\sqrt{\delta}]$ \quad sinon,
\end{itemize}
pour tout entier n.
\end{rem}

\begin{proof}[Preuve de la proposition~\ref{Hunit}]
	$\underline{\Longleftarrow}$
	En utilisant la proposition~\ref{FQ}, on obtient que
	\[ \discr(BA^nC\transpA^n) = \Tr^2(HA^n)(\Tr^2(HA^n) + 4 \epsilon) \]
	pour un $\epsilon \in \{-1, 1\}$, et on peut choisir $\epsilon$ comme l'on veut quitte \`a transposer $C$.
	
	L'entier $\Tr(U^{an+b})$ est solution $x$ d'une \'equation de Pell-Fermat $x^2-\delta y^2 = \pm4$. En effet, on a d'une part que le discriminant de $U^{an+b}$ est de la forme $\discr(U^{an+b}) = \delta y^2$ puisque le corps de $U$ est $\Q[\sqrt{\delta}]$, et on a d'autre part les \'egalit\'es $\discr(U^{an+b}) = \Tr(U^{an+b})^2 - 4\Det(U^{an+b})$, et $\Det(U^{an+b}) = \pm 1$.
	
	On obtient donc, quitte \`a transposer $C$, que $\discr(BA^nC\transpA^n)$ est de la forme $\delta z^2$, donc que le corps de $BA^nC\transpA^n$ est $\Q[\sqrt{\delta}]$.
	
	$\underline{\Longrightarrow}$
	Pour d\'emontrer cette implication, on reprend les choses qui ont \'et\'e faites pour \'etudier les suites de la forme $AB^n$ (voir chapitre~\ref{suites1}).
	Soit $U$ la matrice donn\'ee par le lemme~\ref{pell}.
	On a que $\Tr(HA^n)$ est solution $x$ d'une \'equation de Pell-Fermat $x^2-\delta y^2 = \pm 4$, donc d'apr\`es le lemme~\ref{pell}, il existe une suite d'entiers $(i_n)$ telle que pour tout $n$, on ait $\Tr(HA^n) = \Tr(U^{i_n})$.
	On est alors exactement dans la situation du lemme~\ref{arithm}, et on obtient alors que la suite $(i_n)$ est arithm\'etique : il existe donc deux entiers $a$ et $b$ tels que pour tout $n$, $i_n = an+b$.
	D'apr\`es le lemme~\ref{infini-toujours}, on a de plus que les matrices $A$ et $U^a$ sont semblables (i.e. ont m\^emes valeurs propres), donc le corps de $A$ est $\Q[\sqrt{\delta}]$.
\end{proof}



La proposition suivante permet de ramener la recherche de matrices non inversibles  et \`a coefficients entiers $H$ telles que $\Tr(HA)$ est solution $x$ enti\`ere de l'\'equation de Pell-Fermat $x^2-\delta y^2 = \pm4$ ($A$ \'etant fix\'ee), \`a la recherche de matrices $S$ v\'erifiant certaines propri\'et\'es plus simples :

\begin{prop} \label{HS}
	Soit $A$ une matrice positive, et $b$ un entier.
	Les assertions suivantes sont \'equivalentes :
	\begin{enumerate}
		\item	Il existe une matrice $H \in M_2(\Z)$ de rang 1, telle que pour tout $n \in \Z$,
		\[
			\Tr(HA^n) = \Tr(A^{n+b}).
		\]
		\item	Il existe une matrice $S \in GL_2(\Z)$ telle que $$\Det(S) = -\Det(A^b), \ \Tr(S) = 0 \  \mathrm{et} \  \Tr(SA) = 0.$$
	\end{enumerate}
\end{prop}

\begin{proof}
	$\underline{\Longrightarrow}$
	Soit $S = H - A^b$. On a alors bien $\Tr(S) = 0$ et $\Tr(SA) = 0$. \\
	Montrons que $S \in GL_2(\Z)$.
	Si l'on diagonalise $A$, on voit que la condition
	\[
		\Tr(HA^n) = \Tr(A^{n+b})
	\]
	pour tout entier $n \in \Z$, entra\^ine que dans la base de diagonalisation, $H$ est de la forme : $\left(
	\begin{array}{cc}
	\lambda^b &*\\
	* &\bar{\lambda}^b
	\end{array}
	\right)$, o\`u $\lambda$ et $\bar{\lambda}$ sont les deux valeurs propres de $A$.
	Comme $H$ est non inversible, on obtient l'\'egalit\'e des d\'eterminants $\Det(S) = -\Det(A^b)$. Et comme $S$ est \`a coefficients entiers, on a bien $S \in GL_2(\Z)$.
	
	$\underline{\Longleftarrow}$
	Les \'egalit\'es $\Tr(S) = 0$ et $\Tr(SA) = 0$ entra\^inent d'apr\`es le lemme~\ref{2-toujours}, l'\'egalit\'e $\Tr(SA^n) = 0$ pour tout entier $n$.
	Et par la formule \ref{fm} (page \pageref{fm}) on a
	\[ \Det(S + A^b) = \Det(S) + \Det(A^b) + \Tr(S{A^b}^{\dagger}), \]
	o\`u $M^{\dagger} = \Det(M)M^{-1}$.
	Or on a $\Tr(SA^{-b}) = 0$ et $\Det(S) = -\Det(A^b)$, donc $S + A^b$ est de d\'eterminant 0.
	Finalement $H = S + A^b$ convient.
\end{proof}

\begin{cor} \label{S->suite}
	Soit $A$ une matrice positive, et soit $S$ une matrice de $GL_2(\Z)$ v\'erifiant les conditions $\Tr(S) = \Tr(SA) = 0$ et aussi $\Det(S) = -1$ si $\Det(A) = 1$.
	Alors il existe deux matrices $B$ et $C$ de $\Gamma$ telles que les rangs de $B+\transp{B}$ et de $C+\transp{C}$ sont 1, et telles que les corps des matrices $BA^nC\transpA^n$ sont tous les m\^emes.
\end{cor}

\begin{proof}[Preuve du corollaire]
	L'existence d'une telle matrice $S$ nous donne, par la proposition~\ref{HS} l'existence d'une matrice $H \in M_2(\Z)$ non inversible telle que pour tout entier $n$, $\Tr(HA^n) = \Tr(A^{n+k})$ (pour un $k$ pair si $\Det(S) = \Det(A) = -1$, $k$ impair si $\Det(S) = 1$ et $k$ quelconque sinon). Quitte \`a prendre $k$ assez grand, on peut supposer que les in\'egalit\'es $0 \leq a < b \leq d$ et $a < c \leq d$ sont satisfaites, o\`u $\left(
	\begin{array}{cc}
	a &b\\
	c &d
	\end{array}
	\right) = H$, et donc la proposition~\ref{H->Gamma} nous donne que H s'\'ecrit : $H = X \left(
	\begin{array}{cc}
	0 &0\\
	0 &e
	\end{array}
	\right) Y$, avec $X, Y \in \Gamma$, $e \geq 1$.
	Et comme on a  $a < b$ et $a < c$, on obtient que $X$ et $Y$ sont chacun dans $\Gamma\backslash \{ I_2, \m{1} \}$.
	On peut donc utiliser les relations donn\'ees dans la proposition~\ref{H->Gamma}, pour se ramener \`a $X$ et $Y$ de la forme : $X = X'\m{i}$ et $Y = \m{j}Y'$, avec $i,j \geq 2$ (et on a bien $e \geq 1$).
	D'apr\`es le corollaire~\ref{H}, il existe donc deux matrices positives $B$ et $C$ telles que les rangs de $B+\transp{B}$ et de $C+\transp{C}$ sont $1$ et telles que pour toute matrice $M \in M_2(\R)$ on ait $\Tr(BMC\transp{M}) = (\Tr(HM))^2 - 2\Det(MA^{k})$.
	On a alors $\discr(BA^nC\transpA^n) = \Tr^2(A^{n+k})(\Tr^2(A^{n+k}) - 4\Det(A^{n+k})) = \Tr^2(A^{n+k})\discr(A^{n+k})$, donc pour tout $n$, le corps de $BA^nC\transpA^n$ est le corps de $A$.
\end{proof}

\begin{rem}
	Soient $A, B$ et $C$ trois matrices de $\Gamma$.
	Sous les trois hypoth\`eses 
	\begin{enumerate}
		\item	que $B+\transp{B}$ et $C+\transp{C}$ sont de rang $1$,
		\item	que la matrice $H$ donn\'ee par la proposition~\ref{B+tB} et le corollaire~\ref{H} est enti\`ere,
		\item que l'on peut prendre $U=A$ dans l'\'egalit\'e $\Tr(HA^n) = \Tr(U^{an+b})$ donn\'ee par la proposition~\ref{Hunit},
	\end{enumerate}
	ceci fournit de mani\`ere exhaustive les suites de matrices $BA^nC\transpA^n$ qui ont toutes le m\^eme corps.
	
	La derni\`ere hypoth\`ese est automatiquement satisfaite si $A$ est une matrice donnée par le lemme \ref{pell} (par exemple si $A = \m{1}$), et la premi\`ere hypoth\`ese est v\'erifi\'ee pour tous les exemples connus.
	Il reste des suites infinies \`a \'etudier en retirant l'hypoth\`ese $H$ enti\`ere.
\end{rem}

A l'aide de ce corollaire~\ref{S->suite}, on peut red\'emontrer rapidement le th\'eor\`eme~\ref{MN2} :

\begin{proof}[Preuve du th\'eor\`eme~\ref{MN2}]
	Quitte \`a tout transposer, on peut supposer que c'est $N$ qui est de d\'eterminant $-1$.
	La matrice $S := \left(
	\begin{array}{cc}
	0 &-1\\
	1 &0
	\end{array}
	\right) N$ v\'erifie alors $\Tr(S) = \Tr(SMN) = 0$ (parce que $NMN$ est sym\'etrique), $\Det(S) = -1$ et $S \in GL_2(\Z)$, donc le corollaire~\ref{S->suite} permet de conclure.
\end{proof}

Voici maintenant une preuve plus longue, mais qui fait appara\^itre naturellement la condition sur la matrice $A$ pour que l'on ait une suite $BA^nC\transpA^n$ de matrices ayant toutes pour corps le corps de $A$.

\begin{proof}[Preuve du th\'eor\`eme~\ref{MN2}]
	\'Etant fix\'ee une matrice positive $A$, on cherche une matrice $S$ qui v\'erifie les hypoth\`eses du corollaire~\ref{S->suite}.
	Les condition $\Tr(S) = \Tr(SA) = 0$ imposent de chercher $S$ sous la forme : $S = \left(
	\begin{array}{cc}
	x &y\\
	\cfrac{x(d-a)-cy}{b} &-x
	\end{array}
	\right)$ o\`u $A = \left(
	\begin{array}{cc}
	a &b\\
	c &d
	\end{array}
	\right)$, et on doit avoir $\Det(S) = \pm 1$, ce qui nous donne l'\'equation de Pell-Fermat : $$(2bx + (d-a)y)^2 - \discr(A)y^2 = \pm 4b^2$$
	
	On va maintenant montrer que l'on peut trouver $S$ qui s'exprime simplement en fonction de $A^n$.
	Si l'on d\'efinit les suites $(u_n)$ et $(v_n)$ par :
		$$\left\{
	          \begin{array}{ll}
	            u_0 = 1$ et $u_{n+1} = au_n + bcv_n \\
	            v_0 = 0$ et $v_{n+1} = u_n + dv_n \\
	          \end{array}
	        \right.$$
	alors on a $A^n = \left(
	\begin{array}{cc}
	u_n &bv_n\\
	cv_n &u_n+(d-a)v_n
	\end{array}
	\right)$, et comme on a $\Det(A^n) = \pm 1$, on obtient :
	$$ (2u_n + (d-a)v_n)^2 - \discr(A)v_n^2 = \pm 4 $$
	
	On voit donc que $(x,y) = (u_n, bv_n)$ fournit une solution.
	La derni\`ere chose \`a v\'erifier est que la matrice $S$ trouv\'ee est \`a coefficients entiers.
	Or, cela est le cas si et seulement si $b$ divise $d-a$, puisque $b$ divise $y$ et $b$ est premier \`a $x$.
	Or, les matrices $A \in \Gamma$ qui v\'erifient cette relation de divisibilit\'e sont exactement les matrices de la forme $\m{k}M$ avec $k \geq 1$ et $M \in \Gamma$ sym\'etrique. 
	Ainsi on a bien démontré le théorème~\ref{MN2}. Et le théorème~\ref{MN} lui est équivalent (voir remarque~\ref{rem_MN2}).
\end{proof}

\section{Exemples} \label{exs}

Dans ce chapitre, nous d\'ecrivons pr\'ecis\'ement les suites de fractions continues p\'eriodiques que donne la preuve du th\'eor\`eme~\ref{MN}, et nous donnons des exemples, que nous vérifions directement. \\

Etant donn\'es une matrice $A$ positive fix\'ee, un entier $k$ et une matrice $H$ non inversible telle que pour tout $n$, $\Tr(HA^n) = \Tr(A^{n+k})$, le corollaire~\ref{H} et les propositions~\ref{H->Gamma} et \ref{FQ} permettent de trouver toutes les matrices $B$ et $C$ telles que 
\[ \text{ pour toute matrice } M \in M_2(\R), \Tr(BMC\transp{M}) = \Tr(HM)^2 \pm 2 \Det(M). \]
Cela donne alors une suite de matrices $BA^nC\transpA^n$ ayant toutes le m\^eme corps.

On peut d\'eterminer explicitement quelles sont les matrices $H$ qui sont donn\'ees par la proposition~\ref{HS}, \`a partir des matrices $S$ choisies ci-dessus dans la preuve du th\'eor\`eme, pour $N = \m{i}$.
En prenant $A = M\m{i}$ pour $M \in \Gamma$ sym\'etrique, $b=2$ et $S = S_0\m{i}$ dans la proposition~\ref{HS}, on a $H = S_0\m{i} + M\m{i}M\m{i}$, donc on obtient :
\[ H = M(\Det(M)S_0 + \m{i})M\m{i}
= \left\{ \begin{array}{cl}
	MRM\m{i} & \text{si} \quad \Det(M) = 1 \\
	M\transp{R}M\m{i} & \text{si} \quad \Det(M) = -1
 \end{array} \right. \]
 o\`u $R := \left(
		\begin{array}{cc}
			0 &0\\
			2 &i\\
		\end{array}
		\right)$.

\noindent Et l'on peut expliciter la décomposition de la matrice $R$ donn\'ee par la proposition~\ref{H->Gamma} :

\begin{tabular}{ll}
	$R = \left(
		\begin{array}{cc}
			0 &0\\
			0 &2\\
		\end{array}
		\right) \m{\frac{i}{2}} \quad$ & si $i$ pair, \\
	$R = \left(
		\begin{array}{cc}
			0 &0\\
			0 &1\\
		\end{array}
		\right)\m{2}\m{\frac{i-1}{2}} \quad$ & si $i$ impair.
\end{tabular}

\noindent On choisit alors des matrices $B$ et $C$ de $\Gamma$ qui correspondent, par le corollaire~\ref{H}, aux matrices $H$ de rang 1 ci-dessus.
Par exemple, si $\Det(M) = 1$, si $i$ est pair, et si $M = \m{j }S\m{ j}$ pour une matrice sym\'etrique $S$ de $\Gamma$, on peut prendre :
\[
	\begin{tabular}{lll}
		$B = \m{(i, j )}S\m{(j, i/2-1, 1, 1, i/2, j )}S\m{ (j, i)}$, \\
		$C = \m{j}S\m{(j-1, 1, 1, j )}S\m{ j}$.
	\end{tabular}.
\]

\noindent On obtient finalement les suites de fractions continues p\'eriodiques suivantes :

\begin{prop} \label{suitesijMj}
	Soit $\pat{M} = (a_1, a_2, ... , a_2, a_1)$ un uplet sym\'etrique d'entiers strictement positifs et soient $i$ et $j$ deux entiers strictement positifs.
	
	Si $i$ est pair, alors pour tout entier $n$,
	$$ [ \overline{ i/2-1, 1, 1, i/2, i^n, i-1, 1, 1, i, i^n } ] \in \Q[\sqrt{i^2+4}]$$
	$$ [ \overline{ i/2-1,1,1,i/2,(j,i)^n,j-1,1,1,j,(i,j)^n } ] \in \Q[\sqrt{(ij)^2+4ij}]$$
	$$ [ \overline{ i/2-1, 1, 1, i/2, (j, \pat{M}, j, i)^n, j, \pat{M}, j-1, 1, 1, j, \pat{M}, j, (i, j , \pat{M}, j)^n } ]$$
	$$\in \Q[\sqrt{\discr(\m{(j, a_1, a_2, ..., a_2, a_1, j, i)})}]. $$ 
	
	Et si $i$ est impair, alors pour tout $n$,
	$$ [ \overline{ (i-1)/2, 1, 3, (i-1)/2, i^n, i+1, i-1, i^n } ] \in \Q[\sqrt{i^2+4}]$$
	$$ [ \overline{ (i-1)/2, 1, 3, (i-1)/2, (j, i)^n, j+1, j-1, (i, j)^n } ] \in \Q[\sqrt{(ij)^2+4ij}]$$
	$$ [ \overline{ (i-1)/2, 1, 3, (i-1)/2, (j, \pat{M}, j, i)^n, j, \pat{M}, j+1, j-1, \pat{M}, j, (i, j , \pat{M}, j)^n } ] $$
	$$\in \Q[\sqrt{\discr(\m{(j, a_1, a_2, ..., a_2, a_1, j, i)})}]. $$
	
	\noindent Et l'on obtient de vraies fractions continues p\'eriodiques m\^eme s'il y a des entiers nuls, en utilisant la relation $\m{(i, 0, j)} = \m{i+j}$.
\end{prop}
Ici, $i^n$ signifie que l'entier $i$ est r\'ep\'et\'e $n$ fois, et de m\^eme $( j,\pat{M}, j, i )^n$ signifie que le motif $ j, a_1, a_2, ... , a_2, a_1, j, i $ est r\'ep\'et\'ee $n$ fois.

\begin{ex}
	En choisissant dans la proposition pr\'ec\'edente $i =2$, $j=2$ et $\pat{M} = (1,1,1,1,2,1,1,1,1,1,2,1,1,1,1)$, on obtient pour tout entier $n$ la fraction continue p\'eriodique
	$$ [ \overline{ 1, 1, 1, 2, (\pat{M},2,2,2)^n, \pat{M}, 3, 1, 1, 2, \pat{M}, (2, 2, 2, \pat{M})^n } ] \in \Q[\sqrt{2}]$$
	de longueur $36n+38$, et n'ayant que des $1$ et des $2$, \`a l'exception d'un $3$.
\end{ex}

\begin{proof}[Vérification]
	La proposition~\ref{suitesijMj} donne la suite de fractions continues p\'eriodiques :
	$$ [ \overline{0,1,1,1,(2,\pat{M},2,2)^n,2,\pat{M},1,1,1,2,\pat{M},2,(2,2,\pat{M},2)^n} ] $$
	qui devient par permutation circulaire et décalage des puissances :
	$$ [ \overline{1,1,1,2,(\pat{M},2,2,2)^n,\pat{M},2,0,1,1,1,2,\pat{M},(2,2,2,\pat{M})^n} ] $$
	et on obtient enfin la suite annonc\'ee en utilisant la relation $\m{(2, 0, 1)} = \m{3}$ (ce qui revient \`a remplacer le motif 2,0,1 par 3). \\
	Il suffit alors de vérifier que le corps de la matrice $A = \m{(2, 2, 2, 1, 1, 1, 1, 2, 1, 1, 1, 1, 1, 2, 1, 1, 1, 1)}$ est $\Q[\sqrt{2}]$.
	On a $A = \left( \begin{array}{cc} 7918 & 12929 \\ 19159 & 31284 \end{array} \right)$, donc $\discr(A) = 1 536 796 800 = 2 \times 27 720^2$.
\end{proof}


\begin{rem}
	Il est possible de r\'e\'ecrire matriciellement sous une forme plus simple les suites de fractions continues donn\'ees dans la proposition~\ref{suitesijMj} : \\
	
	Le corps de la matrice $S_1(M\m{i})^nS_2(M\m{i})^n$ est $\Q[\sqrt{\discr(M\m{i})}]$, \\
	
	\noindent pour toute matrice sym\'etrique $M$, pour tout entier $i \geq 1$, et pour tout entier $n$,
	o\`u $S_1$ et $S_2$ sont deux sym\'etries de $GL_2(\Z)$, donn\'ees par 
	
	\noindent $$\left\{\begin{array}{ll}
			S_1 = {\m{i}}^{-1}\m{i/2-1}\m{1}\m{1}\m{i/2} = \left(
			\begin{array}{cc}
				-i-1 & -i^2/2-i\\
				2  & i+1\\
			\end{array}
			\right) \\
			S_2 = \m{i}^{-1}\m{j}^{-1}\m{j-1}\m{1}\m{1} = \left(
			\begin{array}{cc}
				1 &2+i\\
				0 &-1\\
			\end{array}
			\right)
			\end{array}\right\}
			\text{si $i$ est pair},$$
			
	\noindent $$\left\{\begin{array}{ll}
			S_1 = \m{i}^{-1}\m{(i-1)/2}\m{1}\m{3}\m{(i-1)/2} = \left(
			\begin{array}{cc}
				-2i+1 & -i^2+i\\
				4 & 2i-1\\
			\end{array}
			\right) \\
			S_2 = \m{i}^{-1}\m{j}^{-1}\m{j+1}\m{j-1}\m{j}^{-1} = \left(
			\begin{array}{cc}
				-1 &1-i\\
				0 &1\\
			\end{array}
			\right)
			\end{array}\right\} \text{sinon},$$
	o\`u $j$ est un entier quelconque.
\end{rem}

La proposition suivante nous donne l'existence, dans tout corps quadratique r\'eel, d'une infinit\'e de fractions continues p\'eriodiques qui n'utilisent que trois entiers diff\'erents.

\begin{prop} \label{suite-1-2-s}
	Soit $\delta$ un entier non carr\'e.
	Alors il existe un entier $s \geq 1$ tel que pour tout entier $n$,
	$$ [ \overline{ 2,1,1,1, (s, 1,1,2,1,1)^n, s,1,2,1,1,1,1,s, (1,1,2,1,1,s)^n } ] \in \Q[\sqrt{\delta}]. $$
\end{prop}

\begin{cor}
	Pour tout corps quadratique $\Q[\sqrt{\delta}]$, il existe un réel $m_\delta$ et une infinit\'e de fractions continues p\'eriodiques $[ \overline{ a_0, a_1, ..., a_n } ]  \in \Q[\sqrt{\delta}] $ avec $1 \leq a_i \leq m_{\delta}$.
\end{cor}

Voici des exemples de suites de fractions continues qui ne comportent que les entiers 1 et 2 et qui restent dans un corps quadratique donné :

\begin{ex} \label{suite-7}
	Pour tout entier $n$, la fraction continue p\'eriodique suivante est dans $\Q[\sqrt{7}]$ :
	$$ [ \overline{ 2,1,1,1, (1,1,1,1,1,1,1,2,1,2)^n, 1,1,1,1,2,1, (2,1,2,1,1,1,1,1,1,1)^n } ]. $$
\end{ex}

\begin{prop} \label{suite-sym}
	Pour tout uplet sym\'etrique d'entiers $\pat{S} = (a_1, a_2, ... , a_2, a_1)$, et pour tout entier $n$, la fraction continue p\'eriodique
	$$ [ \overline{ 2,1,1,1, (\pat{S}, 1,1,2,1,1)^n,\pat{S},1,2,1,1,1,1,\pat{S}, (1,1,2,1,1,\pat{S})^n } ] $$ est dans $\Q[\sqrt{\discr(\m{(a_1, a_2, ..., a_2, a_1, 1, 1, 2, 1, 1)})}]$.
\end{prop}

En particulier en choisissant $\pat{S}=1^{n-5}$ (c'est-\`a-dire l'entier 1 r\'ep\'et\'e $n-5$ fois), on a des suites de fractions continues p\'eriodiques form\'ees seulement des entiers $1$ et $2$ dans $\Q[\sqrt{f_n f_{n+2}}]$, pour tout $n \geq 3$, o\`u $(f_n)_{n \in \N}$ est la suite de Fibonnacci, d\'efinie par $f_0 = 0$, $f_1 = 1$ et pour tout entier n positif ou nul $f_{n+2} = f_{n+1}+f_{n}$. Cela nous donne une infinit\'e de corps quadratiques d'apr\`es le lemme suivant :

\begin{lemme}
	L'ensemble des corps quadratiques $\Q[\sqrt{f_nf_{n+2}}]$ est infini.
\end{lemme}

\begin{proof}
	Il est bien connu que pour tout $n$, les entiers $f_n$ et $f_{n+2}$ sont premiers entre eux.
	Montrons que pour tout nombre premier $p$, il existe un entier $n \geq 1$ tel que $p$ divise $f_n$.
	La matrice $\m{1}^n$ s'\'ecrit :  $\m{1}^n = \left(
			\begin{array}{cc}
				f_{n-1} & f_n\\
				f_n & f_{n+1}\\
			\end{array}
			\right)$, et si l'on note $o$ l'ordre du groupe fini multiplicatif $GL_2(\Z/{p\Z})$, on a $\m{1}^o \equiv I_2$ (mod $p$), donc $f_o$ est divisible par $p$.
	Utilisons alors la propri\'et\'e suivante des nombres de Fibonacci :
	
	\begin{props}
		Soit $p$ un nombre premier impair, et soit $n \geq 1$ un entier.
		Si $q = p^k$ est la plus grande puissance de $p$ divisant $f_n$, avec $k\geq 1$, alors $pq$ est la plus grande puissance de $p$ divisant $f_{pn}$.
	\end{props}
	
	\begin{proof}
		Comme $f_{n+1} = f_n + f_{n-1}$, on a $\m{1}^n \equiv cI_2$ (mod $q$) pour $c = f_{n-1}$.
		On peut donc \'ecrire $\m{1}^n = cI_2 + qA$, pour une matrice $A \in M_2(\Z)$.
		On a alors
		$$\m{1}^{pn} = (cI_2 + qA)^p = \sum_{i=0}^{p}{ \binom{p}{i} c^i(qA)^{p-i}} \equiv c^pI_2 + pc^{p-1}qA \quad (\mathrm{mod} \  p^2q)$$
		puisque $p^2q$ divise $\binom{p}{i} c^i(qA)^{p-i}$ d\`es que $i<p-1$.
		Donc on obtient $f_{np} \equiv pc^{p-1}f_n$ (mod $p^2q$), et comme $c$ est premier \`a $p$,
		cela donne bien que la plus grande puissance de $p$ divisant $f_{np}$ est $qp$.
	\end{proof}
	
	Ce dernier lemme permet d'obtenir, pour tout nombre premier $p$ impair, un entier $n$ pour lequel $p$ divise le facteur sans carr\'e de $f_n$ et donc aussi le facteur sans carr\'e de $f_nf_{n+2}$.
	On en d\'eduit que l'ensembles des corps quadratiques $\Q[\sqrt{f_nf_{n+2}}]$ est infini.
\end{proof}

La proposition~\ref{suite-sym} est une cons\'equence imm\'ediate de la proposition~\ref{suitesijMj}, mais voici une vérification directe :

\begin{proof}[Preuve de la proposition~\ref{suite-sym}] 
	Soient $S \in \Gamma$ symétrique, $A = S \m{(1, 1, 2, 1, 1)}$, $B = \m{(2, 1, 1, 1)}$ et $C = S\m{(1, 2, 1, 1, 1, 1)}S$.
	Soit alors $H = S\m{1}H_0$, la matrice donn\'ee par le corollaire~\ref{H}, o\`u $H_0 = \m{2} \left(
	\begin{array}{cc}
	0 & 0\\
	0 & 2
	\end{array}
	\right) \m{2} = \left(
	\begin{array}{cc}
	2 &4\\
	4 &8
	\end{array}
	\right)$.
	
	Pour obtenir le r\'esultat, il suffit de montrer que $\Tr(\m{1}H_0M) = \Tr(\m{(1, 1, 2, 1, 1)}M)$, pour toute matrice $M$ sym\'etrique.
	En effet, en prenant $M = (S\m{(1, 1, 2, 1, 1)})^nS$, on obtient alors $\Tr(HA^n) = \Tr(A^{n+1})$ qui donne bien le r\'esultat d'apr\`es la proposition~\ref{FQ} : 
	$$
	\begin{aligned}
		\discr(BA^nC\transp{A}^n) &= \Tr^2(HA^n)(\Tr^2(HA^n) + 4\Det(S)\Det(A^n)) \\
							&= \Tr^2(HA)(\Tr^2(A^{n+1}) - 4\Det(A^{n+1})) \\
							&= \Tr^2(HA)\discr(A^{n+1}).
	\end{aligned}
	$$
	Le signe $+\Det(S)$ qui appara\^it devant le terme $\Det(A^n)$ dans la premi\`ere des \'egalit\'es ci-dessus est d\^u au fait que si $B$ et $C$ sont respectivement de la forme $N_B+\epsilon_B S_0$ et $N_C+\epsilon_C S_0$ pour des matrices $N_B$ et $N_C$ sym\'etriques de rang $1$, alors $\epsilon_B \epsilon_C = -\Det(S)$ (voir la remarque~\ref{rqFQ}).
	
	Or, l'\'egalit\'e $\Tr(\m{1}H_0M) = \Tr(\m{(1, 1, 2, 1, 1)}M)$ pour toute matrice $M$ sym\'etrique d\'ecoule simplement de la relation $\m{1} \left(
	\begin{array}{cc}
	2 &4\\
	4 &8
	\end{array}
	\right) - \m{(1, 1, 2, 1, 1)} = \left(
	\begin{array}{cc}
	0 &1\\
	-1 &0
	\end{array}
	\right)$ et du lemme~\ref{tr-sym}.
\end{proof}

\begin{proof}[Vérification de l'exemple~\ref{suite-7}] 
	On a \`a nouveau une suite de la forme $BA^nC\transp{A}^n$ avec cette fois $A = \m{(1, 1, 1, 1, 1, 1, 1, 2, 1, 2)}$, $B = \m{(2, 1, 1, 1)}$ et $C = \m{(1, 1, 1, 1, 2, 1)}$.
	On d\'emontre pour tout $n$ l'\'egalit\'e suivante :
	$$\Tr(H(\m{(1, 1, 1, 1, 1, 1, 1, 2, 1, 2)})^n) = \Tr((\m{(1, 1, 1, 4)})^{2n+1})$$
	o\`u $H = \m{(1, 2)} \left(
	\begin{array}{cc}
	0 & 0\\
	0 & 2
	\end{array}
	\right) \m{2} = \left(
	\begin{array}{cc}
	4 &8\\
	6 &12
	\end{array}
	\right).$
	
	Le corps de la matrice $\m{(1, 1, 1, 4)} = \left(
	\begin{array}{cc}
	2 &9\\
	3 &14
	\end{array}
	\right)$ est bien $\Q[\sqrt{7}]$, donc cela donnera bien le r\'esultat, d'apr\`es la proposition~\ref{Hunit} et la remarque~\ref{rqHunit}, et parce-qu'on est dans le cas o\`u $B$ et $C$ sont toutes les deux des matrices de la forme $N + S_0$, avec $N$ matrice sym\'etrique de rang 1, et o\`u $\Det(\m{(1, 1, 1, 4)}) = 1$. 
	D'apr\`es le lemme~\ref{2-toujours}, comme $\m{(1, 1, 1, 1, 1, 1, 1, 2, 1, 2)} = \left(
	\begin{array}{cc}
	47 & 128\\
	76 & 207
	\end{array}
	\right)$ et $(\m{(1, 1, 1, 4)})^2 = \left(
	\begin{array}{cc}
	31 &144\\
	48 &223
	\end{array}
	\right)$ ont m\^emes valeurs propres (car m\^emes traces et m\^emes d\'eterminants), il suffit de d\'emontrer cette \'egalit\'e pour $n=0$ et $n = 1$.
	Pour $n = 0$, on a $\Tr(H) =  16 = \Tr(\m{(1, 1, 1, 4)})$.
	Pour $n = 1$, on a $\Tr(H\m{(1, 1, 1, 1, 1, 1, 1, 2, 1, 2)}) = 4048 = \Tr((\m{(1, 1, 1, 4)})^3)$.
\end{proof}

\begin{proof}[Preuve de la proposition~\ref{suite-1-2-s}]
	Il suffit de d\'emontrer qu'\'etant donn\'e un entier $\delta$ non carr\'e, il existe un entier $s \geq 1$ tel que le corps de la matrice $\m{(1, 1, 2, 1, 1, s)}$ est $\Q[\sqrt{\delta}]$ (c'est ensuite un cas particulier de la proposition~\ref{suite-sym}).
	Or, on a $\discr(\m{(1, 1, 2, 1, 1, s)}) = 48(s+1)(3s+4)$. Donc il suffit de prendre $s = 3y^2\delta-1$ o\`u $(x,y)$ est une solution \`a l'\'equation de Pell-Fermat $x^2-9\delta y^2 = 1$.
\end{proof}

\subsection{Suites de type Wilson} \label{suites-wilson}

On a vu comment l'on pouvait trouver des suites de fractions continues p\'eriodiques en cherchant les matrices $H$ non inversibles du corollaire~\ref{H} \`a coefficients entiers. Cela revient \`a chercher $\Tr(HA)$ comme solution $x$ enti\`ere de l'\'equation de Pell-Fermat $x^2-\delta y^2 = \pm 4$.
Un autre cas int\'eressant est celui o\`u $H$ est de la forme $\sqrt{\delta}H'$, avec $H'$ enti\`ere, et o\`u $\delta$ est le discriminant de la matrice $A$. On doit alors chercher $\Tr(H'A)$ comme solution $y$ (et non plus $x$) de la m\^eme \'equation de Pell-Fermat, puisque l'on a alors $\discr(BAC\transpA) = \delta (\Tr(H'A))^2(\delta (\Tr(H'A))^2 \pm 4)$, que l'on veut de la forme $\delta x^2$.

Wilson donne des suites qui rentrent dans ce cadre (voir \cite{wilson}), comme par exemple : \\
$$ [ \overline{ (s(s+4)-1), 1, (s, 1)^n, s+2, (s, 1)^n, s+1 } ] \in \Q[\sqrt{s(s+4)}] $$
que l'on peut r\'e\'ecrire avec des entiers plus petits (en choisissant d'autres matrices $B$ et $C$ pour la m\^eme matrice $H$ gr\^ace au corollaire~\ref{H}) : \\
$$ [ \overline{ s, 1, s-1, s+1, (1, s)^n, 1, s+1, s+3, 1, (s, 1)^n } ] \in \Q[\sqrt{s(s+4)}]$$

Cependant, il sera impossible d'obtenir des suites uniform\'ement born\'ees avec une borne ind\'ependante de $\delta$ avec des suites de ce type, puisque si l'on a $H = P \left(
\begin{array}{cc}
0 &0\\
0 &d
\end{array}
\right)Q$ avec $P, Q \in \Gamma$, alors la proposition~\ref{B+tB} et le corollaire~\ref{H} montrent que l'on a n\'ecessairement au moins un facteur $\m{i}$ avec $i > \sqrt{\delta} - 1$ dans la d\'ecomposition de l'une des matrices $B$ ou $C$.

\subsection{R\'eels quasi-palindromiques} \label{rqp}
\`A la vue du th\'eor\`eme \ref{MN}, on peut se demander quels sont les r\'eels quasi-palindromiques, c'est-\`a-dire les r\'eels qui ont un d\'eveloppement en fraction continue p\'eriodique quasi-palindromique. La proposition suivante, r\'epond \`a la question :

\begin{prop}
	Soit $x \in \Q[\sqrt{\delta}]$. On a \'equivalence entre :
	\begin{enumerate}
		\item $x$ est quasi-palindromique,
		\item $\Tr(x) = \floor{x}$ et $x > 1$,
		\item $\Tr(x) \in \Z$, $x > 1$ et $-1 < \overline{x} < 0$.
	\end{enumerate}
\end{prop}

On peut démontrer ce résultat en utilisant un lemme de \'E. Galois sur le miroir d'une fraction continue périodique (voir paragraphe 6, p. 83 dans \cite{perron}).

Les r\'eels $\sqrt{\delta}+\floor{\sqrt{\delta}}$ sont donc des exemples de r\'eels palindromiques. Le r\'esultat classique suivant donne une borne optimale sur le d\'eveloppement de ces r\'eels :

\begin{prop}
	Soit $\delta$ un entier positif sans facteur carr\'e.
	Alors les coefficients du d\'eveloppement en fraction continue du r\'eel $\sqrt{\delta}+\floor{\sqrt{\delta}}$ sont major\'es par $2\floor{\sqrt{\delta}}$.
\end{prop}

Si l'on applique le th\'eor\`eme \ref{MN} avec ce r\'eel $\sqrt{\delta}+\floor{\sqrt{\delta}}$, on obtient donc une infinit\'e de fractions continues p\'eriodiques uniform\'ement born\'ees par $4\floor{\sqrt{\delta}}+1$ dans le corps $\Q[\sqrt{\delta}]$.
Cela améliore le résultat de Wilson, puisque la borne qu'il obtient est seulement en $O(\delta)$.



\section{Conjecture de Zaremba} \label{s_zaremba}

Dans cet article, nous nous sommes intéressé au problème de majorer les coefficients de fractions continues périodiques.
Voici une question similaire à propos des fractions continues finies :

\begin{conj}[Zaremba] \label{conj_zaremba}
	Il existe une constante $m$ telle que pour tout entier $q \geq 1$, il existe un entier $p$ premier à $q$ tel que l'on ait
	\[ \frac{p}{q} = [a_0, a_1, a_2, \dots ] \]
	où les entiers $a_i$ sont entre $1$ et $m$.
\end{conj}

La conjecture semble vraie pour $m = 5$.
Elle semble même vraie pour $m = 2$, mais à condition d'exclure un nombre fini de valeurs de $q$.

Des travaux récents de J. Bourgain et A. Kontorovich vont dans le sens de cette conjecture (voir \cite{bk}).
Plus précisément, ils démontrent que l'ensemble des entiers $q$ qui vérifient la conjecture est de densité $1$ dans $\N$, pour la borne $m=50$.

Nous allons montrer que cette conjecture de Zaremba sur les développements en fractions continues de rationnels implique la conjecture \ref{conjM} sur les développements en fractions continues périodiques d'après le résultat nouveau suivant :

\begin{thm} \label{z_mcm}
	Soient $a$, $b$, $c$ et $\delta$ des entiers strictement positifs tels que
	\begin{itemize}
		\item $b$ et $c$ sont solution de l'équation de Pell-Fermat : $c^2 - \delta b^2 = \pm 1$,
		\item $a$ et $c$ sont premiers entre eux et $a < c$.
	\end{itemize}
	Alors on a l'une des égalités
	\[
		\frac{c-a + b \sqrt{\delta}}{c} = [\overline{1,1,a_1-1, a_2, a_3, \dots, a_{n-1}, a_n, 1, 1, a_n - 1, a_{n-1}, a_{n-2}, \dots, a_2, a_1}]
	\]
	ou
	\[
		\frac{c-a + b \sqrt{\delta}}{c} = [\overline{1,1,a_1-1, a_2, a_3, \dots, a_{n-1}, a_n - 1, 1, 1, a_n, a_{n-1}, a_{n-2}, \dots, a_2, a_1}],
	\]
	où $[0,a_1, a_2, \dots, a_{n-1}, a_n]$ est le développement en fraction continue du rationnel $\frac{a}{c}$.
\end{thm}

Si l'on a des entiers nuls dans la fraction continue donnée par ce théorème (c'est le cas si $a_1 = 1$ ou $a_n = 1$), il suffit de remplacer chaque triplet $x, 0, y$ par $x+y$ pour obtenir une vraie fraction continue. Remarquons aussi que les entiers $\delta$ qui vérifient les hypothèses du théorème ne peuvent pas être carrés.


\begin{proof}[Preuve de Zaremba $\imp$ Conjecture \ref{conjM}]
	Soit $m$ la constante donnée par la conjecture de Zaremba, et soit $\delta$ un entier non carré.
	L'équation de Pell-Fermat $c^2 - \delta b^2 = 1$ admet alors une infinité de solutions $(b, c)$.
	Pour chaque entier $c$ assez grand d'un tel couple solution,
	on choisit alors un numérateur $p$, donné par la conjecture de Zaremba, tel que le développement en fraction continue du rationnel $\frac{p}{c}$ ne s'écrive qu'avec des entiers entre $1$ et $m$. 
	Si l'on choisit pour $a$ le reste de la division de $p$ par $c$, alors le théorème \ref{z_mcm} nous donne une fraction continue périodique bornée par $m+1$ et dans le corps $\Q[\sqrt{\delta}]$.
	On obtient bien une infinité de fractions continues périodiques dans un même corps $\Q[\sqrt{\delta}]$, puisque pour deux rationnels $\frac{a}{c}$ distincts, le théorème donne deux fractions continues périodiques distinctes.
\end{proof}


\begin{proof}[Preuve du théorème \ref{z_mcm}]
	Remarquons que les fractions continues données par le théorème s'écrivent matriciellement sous la forme
	\[
		B A C \transpA,
	\]
	où $A := \m{(a_1, a_2, ..., a_n)}$ est la matrice correspondant au développement du rationnel $\frac{c}{a}$,
	$B := \m{(1,1, a_1-1)} {\m{a_1}}^{-1}$ et $C := \left\{ \begin{array}{l}
											{\m{(1,1, a_n-1)}} {\m{a_n}}^{-1} \\
											\text{ ou } \\
											{\m{a_n}}^{-1} \m{(a_n-1,1,1)}
										\end{array} \right.$.
	
	Un rapide calcul donne $B = \begin{pmatrix}
								0 & 1 \\
								-1 & 2
							\end{pmatrix}$
	et $C = B$ ou $\transp{B}$.
	Nous sommes donc dans le cadre de la proposition \ref{FQ}, avec $H := \begin{pmatrix}
																0 & 0 \\
																0 & 2
															\end{pmatrix}$, ce qui nous donne le discriminant
	\[
		\discr(B A C \transpA) = \Tr^2(HA)(\Tr^2(HA) \pm 4).
	\]
	Or, la matrice $A$ peut s'écrire sous la forme $A = \begin{pmatrix}
												u & a \\
												v & c
											\end{pmatrix}$, pour des entiers $u$ et $v$, puisque c'est la matrice correspondant au développement en fraction continue du rationnel irreductible $\frac{c}{a}$.
	On a donc $\Tr(HA) = 2c$, et donc $\discr(B A C \transpA) = 4c^2 (4 c^2 \pm 4)$.
	En choisissant le signe, quitte à choisir $C = B$ ou $C = \transp{B}$, on a donc finalement
	\[
		\discr(B A C \transpA) = 16 c^2 b^2 \delta,
	\]
	puisque $(b,c)$ est solution de l'équation de Pell-Fermat $c^2 - \delta b^2 = \pm 1$.
	Ainsi, le corps de la matrice $B A C \transpA$ est bien $\Q[\sqrt{\delta}]$ comme annoncé, quitte à transposer $C$.
	
	Pour vérifier l'égalité annoncée, on vérifie que le réel quadratique $x := \frac{c-a + b\sqrt{\delta}}{c}$ correspond (par la proposition \ref{carGamma}) à la matrice
	\[
		B A C \transp{A} = \begin{pmatrix}
							2ac \pm 1 & 2c^2 \\
							4ac - 2a^2 \pm2 & 4c^2 -2ac \pm1
						\end{pmatrix}
	\]
	c'est-à-dire que le vecteur $\begin{pmatrix} 1 \\ x \end{pmatrix}$ est bien un vecteur propre.
\end{proof}

Grâce au théorème \ref{z_mcm}, pour obtenir la conjecture \ref{conjMcMullen} de McMullen, il suffit de démontrer la variante suivante de la conjecture de Zaremba, où l'on impose en plus les premiers et le dernier entiers du développement en fraction continue.

\begin{conj}
	Il existe un entier $q_0$ tel que pour tout entier $q \geq q_0$, il existe un entier $p$ premier à $q$ tel que l'on ait
	\[ \frac{p}{q} = [a_0, a_1, a_2, \dots, a_{n-1}, a_n] \] 
	où les entiers $a_i$ valent chacun $1$ ou $2$, et avec de plus les conditions
	$a_1 = 2$ et $a_n = 2$.
\end{conj}

\begin{rem} \label{r_conjMcM}
	C'est en cherchant, par ordinateur, des fractions continues de la forme donnée par cette conjecture que nous sommes parvenu à vérifier la conjecture~\ref{conj12} pour tous les corps quadratiques réels $\Q[\sqrt{\delta}]$ pour $\delta < 127$, en utilisant le théorème~\ref{z_mcm} ci-dessus.
\end{rem}

%
%
%


En utilisant le résultat de Bourgain et Kontorovich (voir \cite{bk}) sur la conjecture de Zaremba, nous obtenons le résultat suivant : 

\begin{cor} \label{thm_ibk}
	Soit $\D \subseteq \N$ l'ensemble des entiers $\delta$ sans facteurs carrés vérifiant la conjecture~\ref{mcmullen1} pour la borne $m=51$, c'est-à-dire tels que le corps quadratique $\Q[\sqrt{\delta}]$ contienne une fraction continue périodique bornée par $51$.
	Alors on a
	\[
		\# (\D \cap \llbracket  1,N  \rrbracket) \geq \frac{\sqrt{N}}{C},
	\]
	pour une constante $C > 0$ et pour tout entier $N$ assez grand.
\end{cor}

Ce corollaire se déduit facilement des deux lemmes suivants :

\begin{lemme}
	L'ensemble des entiers $n$ tels que le corps $\Q[\sqrt{n^2-1}]$ contienne une fraction continue périodique bornée par $51$ a pour densité $1$ dans l'ensemble des entiers naturels.
\end{lemme}

\begin{proof}
	Soit $n$ un entier vérifiant la conjecture de Zaremba pour la borne $50$, c'est-à-dire tel qu'il existe un numérateur $p < n$ tel que le rationnel $\frac{p}{n}$ soit irréductible et ait un développement en fraction continue borné par $50$.
	D'après Bourgain et Kontorovich (voir \cite{bk}), l'ensemble de ces entiers $n$ est de densité $1$ dans $\N$.
	
	Or, le couple $(n, 1)$ est une solution évidente de l'équation de Pell-Fermat
	\[
		x^2 - (n^2-1) y^2 = 1.
	\]
	Le théorème~\ref{z_mcm} s'applique donc, et on obtient que le réel quadratique
	\[
		\frac{p + \sqrt{n^2-1}}{n} \quad \in \Q[\sqrt{n^2-1}]
	\]
	a un développement en fraction continue borné par $51$. 
\end{proof}

\begin{rem}
	On a un résultat similaire avec les corps $\Q[\sqrt{n^2+1}]$.
\end{rem}

\begin{lemme}
	L'ensemble des entiers $n$ tels que $n^2-1$ soit sans facteur carré a une densité strictement positive dans $\N$.
	C'est-à-dire que l'on a
	\[
		\liminf_{x \to \infty} \frac{1}{x} \# \{ n < x | n^2-1 \text{ sans facteur carré } \} > 0.
	\]
\end{lemme}

\begin{proof}
	Soit
		$\A_d := \{n < x | n^2-1 \text{ est divisible par } d^2 \}.$
	On a alors
	\[
		\liminf_{x \to \infty} \frac{1}{x} \# \{ n < x | n^2-1 \text{ sans facteur carré } \} \geq 1 - \limsup_{x \to \infty} \frac{1}{x} \sum_{\substack{p  \leq \sqrt{x} \\ p \text{ premier}}} \#A_p.
	\]
	Or, on a
	\[
		\# A_p \leq 2\ceil{\frac{x}{p^2}} \leq 2 + \frac{2x}{p^2},
	\]
	et donc
	\[
		 \limsup_{x \to \infty} \frac{1}{x} \sum_{\substack{p  \leq \sqrt{x} \\ p \text{ premier}}} \#A_p \leq  \limsup_{x \to \infty} \frac{2}{\sqrt{x}} + \sum_{p \text{ premier}} \frac{2}{p^2} < 1.
	\]
\end{proof}

\begin{proof}[Preuve du corollaire~\ref{thm_ibk}]
	Les deux lemmes précédents permettent de dire que l'ensemble suivant a une densité strictement positive
	\begin{align*}
		\A := \{ n | & n^2-1 \text{ est sans facteur carré et } \\
		 		& \Q[\sqrt{n^2-1}] \text{ contient une fraction continue périodique bornée par } 51 \}.
	\end{align*}
	On a donc finalement
	\[
		\# (\D \cap \llbracket  1,N  \rrbracket) \geq \# (\A \cap \llbracket  1, \floor{\sqrt{N+1}}  \rrbracket) \geq \frac{\sqrt{N}}{C},
	\]
	pour une constante $C > 0$ et pour tout entier $N$ assez grand.
\end{proof}


\end{document}